\documentclass[a4paper,12pt, reqno]{amsart}
\usepackage{amsfonts, color}
\usepackage[textwidth=460pt]{geometry}
\usepackage{amstext, amsthm, amssymb, amsmath}
\usepackage{mathrsfs}
\usepackage{graphicx}
\usepackage{color}
\usepackage[linesnumbered,lined,ruled,commentsnumbered]{algorithm2e}

\theoremstyle{plain}
\begingroup
\newtheorem{theorem}{Theorem}[section]
\newtheorem{lemma}[theorem]{Lemma}
\newtheorem{proposition}[theorem]{Proposition}
\newtheorem{corollary}[theorem]{Corollary}
\endgroup

\theoremstyle{definition}
\begingroup
\newtheorem{defn}{Definition}

\endgroup

\theoremstyle{remark}
\newtheorem{remark}{Remark}

\numberwithin{equation}{section}

\newcommand{\J}{\mathcal{J}}
\newcommand{\R}{\mathbb{R}}
\newcommand{\F}{\mathcal{F}}
\newcommand{\U}{\mathcal{U}}

\newcommand{\e}{\varepsilon}
\newcommand{\weak}{\rightharpoonup}
\newcommand{\G}{\mathcal{G}}
\newcommand{\xx}{\mathbf{x}}
\newcommand{\bfa}{\mathbf{a}}
\newcommand{\FF}{\mathbf{F}}

\newcommand{\be}{\beta}
\newcommand{\NN}{\mathbb{N}}
\newcommand{\xth}{x^\theta}
\newcommand{\X}{\mathcal{X}}

\begin{document}
\title[Optimal Control of ensembles
of dynamical systems]
{Optimal Control of ensembles of dynamical systems}
\author[A. Scagliotti]{Alessandro Scagliotti}

%\address[A.~Scagliotti]{``School of Computation, Information and Technology'', Technical University of Munich,
%Garching b. M\"unchen,
%Germany \\ \&
%Munich Center for Machine Learning (MCML).}
%\email{scag@ma.tum.de}

\begin{abstract}
In this paper we consider the problem of the optimal control of an ensemble of affine-control systems. After proving the well-posedness of the minimization problem under examination, we establish a $\Gamma$-convergence result that allows us to substitute the original (and usually infinite) ensemble with a sequence of finite increasing-in-size sub-ensembles. The solutions of the optimal control problems involving these sub-ensembles provide approximations in the $L^2$-strong topology of the minimizers of the original problem. 
Using again a $\Gamma$-convergence argument, we manage to derive a Maximum Principle for ensemble optimal control problems with end-point cost. 
Moreover, in the case of finite sub-ensembles, we can address the minimization of the related cost through numerical schemes. In particular, we propose an algorithm that consists of a subspace projection of the gradient field induced on the space of admissible controls by the approximating cost functional. In addition, we consider an iterative method based on the Pontryagin Maximum Principle. Finally, we test the algorithms on an ensemble of linear systems in $\R^2$.

\subsection*{Keywords} Optimal control,
Simultaneous control,
$\Gamma$-convergence, Gradient-based minimization,
Pontryagin Maximum Principle.

\subsection*{Mathematics Subject Classification}
49J15, 49K15, 49M05.

\end{abstract}

\maketitle

\subsection*{Acknowledgments} 
A great part of the work presented here was done while the Author was a Ph.D. candidate at Scuola Internazionale Superiore di Studi Avanzati (SISSA), Trieste, Italy. The Author acknowledges partial support from INdAM--GNAMPA.
The Author thanks Prof. Andrei Agrachev for encouragement and helpful discussions.
Finally, the Author wants to express his gratitude to two anonymous Referees, whose comments helped to improve the quality of the paper. In particular, the results presented in Section~\ref{sec:PMP_ens} were inspired by the observation of a Reviewer.

\section*{Introduction}
\label{sec:Intro}
An ensemble of control systems is a parametrized family of controlled ODEs of the form 
\begin{equation} \label{eq:general_ens_intro}
\begin{cases}
\dot x^\theta(t) = G^\theta(x^\theta(t),u(t))
& \mbox{a.e. in }[0,T],\\
x^\theta(0) = x_0^\theta,
\end{cases}
\end{equation}
where $\theta\in \Theta\subset \R^d$ is the parameter of the ensemble, $u:[0,T]\to \R^k$ is the control, and, for every $\theta\in \Theta$, $G^\theta:\R^n\times\R^k\to \R^n$ is the function that prescribes the dynamics of the corresponding system.
The peculiarity of this kind of problem is that the elements of the ensemble are \textit{simultaneously} driven by the same control $u$. This framework is particularly suitable for modeling real-world control systems affected by data uncertainty (see, e.g., \cite{RuLi12}), or the problem of controlling a large number of particles through a signal (see \cite{Bro09}). Also from the theoretical viewpoint, there is currently an active research interest in this topic. For instance, the problem of the controllability of ensembles of linear equations has been recently investigated in \cite{DS21}. In \cite{ABS16} it was proved a generalization of the Chow–Rashevskii theorem for ensembles of linear-control systems. In \cite{LK06,LK09} ensembles were studied in the framework of nuclear magnetic resonance spectroscopy. Moreover, as regards ensembles in quantum control, we report the contributions \cite{BCR10,BST15}, and we recall the recent works \cite{AuBoSi18,ChiGau18}. Finally, we mention that  the interplay between \textit{Reinforced Learning} and optimal control of systems affected by partially unknown dynamics has been 
investigated in \cite{MP18, PPF21, PPF21_2, PPF22}.

In the present paper, we focus on a particular instance of \eqref{eq:general_ens_intro}, corresponding to the case in which the dynamics has an affine dependence on the controls. More precisely, we consider ensembles with the following expression: 
\begin{equation} \label{eq:aff_ctrl_ens_intro}
\begin{cases}
\dot x^\theta (t) = F_0^\theta(\xth(t)) + F^\theta(
\xth(t))u(t) & \mbox{a.e. in }[0,1],\\
\xth(0) = \xth_0,
\end{cases}
\end{equation}
where $\theta \in \Theta\subset \R^d$ varies in a compact set, and, for every $\theta\in \Theta$, the vector field $F_0^\theta:\R^n\to\R^n$ represents the drift, while the matrix-valued application $F^\theta=(F_1^\theta,\ldots,F_k^\theta):\R^n\to\R^{n\times k}$ collects the controlled fields. We set $\U:=L^2([0,1],\R^k)$ as the space of admissible controls, and, for every $\theta\in \Theta$, the curve $x_u^\theta:[0,1]\to\R^n$ denotes the trajectory of \eqref{eq:aff_ctrl_ens_intro} corresponding to the parameter $\theta$ and to the control $u\in\U$. We are interested in the optimal control problem related to the minimization of a functional $\F:\U\to\R_+$ of the form
\begin{equation} \label{eq:intro_ens_fun}
\F(u):= \int_{\Theta}\int_0^1 
a(t,x_u^\theta(t),\theta) \,d\nu(t)d\mu(\theta)
+ \frac\beta2 ||u||_{L^2}^2
\end{equation} 
for every $u\in\U$, where $a:[0,1]\times\R^n\times\Theta\to\R_+$ is a non-negative continuous function, while $\nu,\mu$ are Borel probability measures on $[0,1]$ and $\Theta$, respectively, and  $\beta>0$ is a constant that tunes the $L^2$-squared regularization.
When the support of the probability measure $\mu$ is not reduced to a finite set of points, the minimization of the functional $\F$ is often intractable in practical situations since a single evaluation of $\F$ potentially requires the resolution of an infinite number of Cauchy problems\eqref{eq:aff_ctrl_ens_intro}. Therefore, it is natural to try to replace $\mu$ with a sequence of probability measures $(\mu_N)_{N\in\NN}$ such that each of them charges a finite subset of $\Theta$, and such that $\mu_N\weak^* \mu$ as $N\to\infty$. Then, we can consider the sequence of functionals $(\F^N)_{N\in\NN}$ defined as 
\begin{equation}\label{eq:intro_def_fun_ens_N}
\F^N(u):= \int_{\Theta}\int_0^1 
a(t,x_u^\theta(t),\theta) \,d\nu(t)d\mu_N(\theta)
+ \frac\beta2 ||u||_{L^2}^2
\end{equation}
for every $u\in \U$ and for every $N\in\NN$.
One of the goals of the present work is to study in which sense the functionals defined in \eqref{eq:intro_def_fun_ens_N} approximate the cost $\F$. It turns out that, when considering the restrictions to bounded subsets of $\U$, the sequence $(\F^N)_{N\in\NN}$ is $\Gamma$-convergent to $\F$ with respect to the {\it weak} topology of $L^2$. We report that a similar approach was undertaken in \cite{PRG}, where the authors considered ensembles of control systems in the general form \eqref{eq:general_ens_intro}, and it was proved that the \textit{averaged approximations} of the cost functional under examination are $\Gamma$-convergent to the original objective with respect to the \textit{strong} topology of $L^2$. We insist on the fact that our result is not reduced to a particular case of the one studied in \cite{PRG}. Indeed, on one hand, using the \textit{strong} topology, in \cite{PRG} it was possible to establish $\Gamma$-convergence for more general ensembles of control systems, and not only under the affine-control dynamics \eqref{eq:aff_ctrl_ens_intro}. On the other hand, in the general situation considered in \cite{PRG} the functionals of the approximating sequence are not equi-coercive (often neither coercive) in the $L^2$-\textit{strong} topology, and proving that the minimizers of the approximating functionals are (up to subsequences) convergent could be a challenging task.
However, in the case of affine-control systems we manage to prove $\Gamma$-convergence even if the space of admissible controls $\U$ is equipped with the {\it weak} topology. Moreover, if for every $N\in\NN$ we choose $u_N\in \arg\min_\U \F^N$, standard facts in the theory of $\Gamma$-convergence ensure that the sequence $(u_N)_{N\in\NN}$ is \textit{weakly} pre-compact and that each of its limiting points is a minimizer of the original functional $\F$ defined in \eqref{eq:intro_ens_fun}. What is more surprising is that --owing to the peculiar form of the cost \eqref{eq:intro_ens_fun}-- it turns out that $(u_N)_{N\in\NN}$ is also pre-compact in the $L^2${\it-strong} topology. Similar phenomena have been recently observed in \cite{S1} and \cite{S2}, respectively in the frameworks of sub-Riemannian geodesics approximations and of data-driven diffeomorphisms reconstruction.
%Motivated by this fact, we finally formulate
%two algorithms for the numerical 
%minimization of the functionals $(\F^N)_{N\in\NN}$.

In the second part of the paper, we restrict our  focus to the case of the average end-point cost, i.e., when  $\nu = \delta_{t=1}$ in the integral at the right-hand side of \eqref{eq:intro_ens_fun} and \eqref{eq:intro_def_fun_ens_N}.
In this framework, from a direct application of the classical theory, we first derive  the Pontryagin Maximum Principle for the problem of minimizing the functional $\F^N$ for $N\in\NN$. Then, using again an argument based on the $\Gamma$-convergence, we manage to formulate the Pontryagin necessary conditions for local minimizers of the functional $\F$. We report that our analysis has been inspired by the results in \cite{BK19}, where the authors establish the Maximum Principle for a large class of ensemble optimal control problems with average end-point cost. Even though our strategy is analogous to the path described in \cite{BK19} (i.e., first considering auxiliary problems involving discrete measures, and then recovering the Maximum Principle for the ensemble optimal control problem), our case is not covered by the results presented in \cite{BK19}. Namely, in \cite{BK19} it is required that, for every point in a neighborhood of an optimal trajectory, the set of the admissible velocities is bounded, and this fact is crucial to prove the continuity of the trajectories when the controls are equipped with the Ekeland metric (see \cite[Lemma~5.1]{BK19}). Moreover, we observe that in \cite{BK19} the limiting process evokes Ekeland's variational principle, while we employ $\Gamma$-convergence.
Finally, we recall that in \cite{V05} the Maximum Principle for \textit{minimax} optimal control was derived.

In the last part, we propose two numerical schemes for finite-ensemble optimal control problems with average end-point cost. 
More precisely, recalling that $\U$ is endowed with the usual Hilbert space structure, we first consider the gradient field induced by the functional $\F^N:\U\to\R_+$ on its domain. This is done by adapting to the affine-control case a result obtained in \cite{S1} for linear-control systems. Then, we construct Algorithm~\ref{alg:proj_grad_field} as the orthogonal projection of this gradient field onto a subspace $\U_M\subset \U$ such that $\dim(\U_M)<\infty$.
On the other hand, Algorithm~\ref{alg:iter_PMP} is an adaptation to our problem of an iterative scheme originally proposed in \cite{SS80}, based on the Maximum Principle. Variants of Algorithm~\ref{alg:proj_grad_field} and Algorithm~\ref{alg:iter_PMP} have been recently introduced in\cite{S2} as training procedures of a control-theoretic inspired Deep Learning architecture.
We recall that a multi-shooting technique for ensemble optimal control has been recently investigated in \cite{LGV20}.

We briefly outline the structure of this work.\\
In Section~\ref{sec:prel} we establish some preliminary results. In particular, we show that the trajectories of the ensemble \eqref{eq:aff_ctrl_ens_intro} are uniformly $C^0$-stable for $L^2$-weakly convergent sequences of admissible controls. This property is peculiar to affine-control dynamics and plays a crucial role in the other sections.\\
%In Subsection~\ref{subsec:aff_grad_field} we 
%derive the expression of the gradient field
%on the space of controls
%induced by an optimal control problem 
%with end-point cost.
In Section~\ref{sec:opt_ctrl_ens} we formulate the ensemble optimal control problem related to the minimization of the functional $\F:\U\to\R_+$ defined in \eqref{eq:intro_ens_fun}, and we prove the existence of a solution using the direct method of calculus of variations. \\
In Section~\ref{sec:G_conv} we establish the approximation results by showing that the sequence of functionals $(\F^N)_{N\in\NN}$ defined as in \eqref{eq:intro_def_fun_ens_N} are $\Gamma$-convergent to $\F$ with respect to the weak topology of $L^2$.\\
In Section~\ref{sec:grad_flow}, for every $N\in\NN$, we compute the gradient field induced by the functional $\F^N$ on the space of admissible controls, and we derive the Pontryagin Maximum Principle for the optimal control problem related to the minimization of $\F^N$. Starting from Section~\ref{sec:grad_flow} we restrict our attention to the end-point integral cost, that corresponds to the choice $\nu=\delta_{t=1}$ in \eqref{eq:intro_ens_fun}. \\
In Section~\ref{sec:PMP_ens} we prove the Maximum Principle for local minimizers of the functional $\F$, using a strategy based on $\Gamma$-convergence and the construction of auxiliary problems involving finite ensembles of control systems.  \\
In Section~\ref{sec:num_schemes} we construct two numerical schemes for the minimization of $\F^N$ in the case of end-point cost. The first method is based on the gradient field derived in Section~\ref{sec:grad_flow}, while for the second we make use of the Maximum Principle for finite ensembles.\\
Finally, in Section~\ref{sec:num_exp} we test the algorithms on an \textit{approximately controllable} ensemble of systems in $\R^2$.

\subsection*{General Notations}
We introduce below some basic notations.
For every $d\geq 1$, we consider the space 
$\R^d$ endowed with the usual Euclidean
norm $|z|_2:=\sqrt{\langle z,z\rangle_{\R^d}}$
for every $z\in\R^d$,
induced by the scalar product
\[
\langle x,y \rangle_{\R^d}:=
\sum_{i=1}^d x_iy_i \qquad
\]
for every $x=(x_1,\ldots,x_d),y=(y_1,\ldots,y_d)
\in \R^d$.
We sometimes make use of the equivalent norm
$|\cdot|_1$ defined as $|z|_1=\sum_{i=1}^d|z_i|$
for every $z\in\R^d$. We recall that the inequality
\begin{equation} \label{eq:norm_eucl_equiv}
\frac1{\sqrt{d}} |z|_2\leq |z|_1 \leq \sqrt{d}|z|_2
\end{equation}
holds for every $z\in\R^d$.

\begin{section}{Framework and Preliminary results}\label{sec:prel}
In this paper, we study ensembles of control systems in $\R^n$ with affine dependence in the control variable $u\in \R^k$. More precisely, given a compact set $\Theta$ embedded into a finite-dimensional Euclidean space, for every $\theta\in\Theta$ we are assigned an affine-control system of the form
\begin{equation} \label{eq:ens_ctrl_sys}
\begin{cases}
\dot x^\theta (t) = F_0^\theta(\xth(t)) + 
F^\theta( \xth(t))u(t),\\
\xth(0) = \xth_0,
\end{cases}
\end{equation} 
where for every $\theta\in\Theta$ we require that $F_0^\theta:\R^n\to\R^n$ and $F^\theta:\R^n\to \R^{n\times k}$ are Lipschitz-continuous applications. We stress the fact that the control $u:[0,1]\to\R^k$ does not depend on $\theta$, so it is the same for every control system of the ensemble. 
Let us introduce $F_0:\R^n\times \Theta\to\R^n$ and $F:\R^n\times \Theta \to \R^{n\times k}$ defined respectively as 
\begin{equation} \label{eq:ens_fields}
F_0(x,\theta) := F_0^\theta(x) \qquad\mbox{and}\qquad
F(x,\theta) := F^\theta(x)
\end{equation}
for every $(x,\theta)\in \R^n\times\Theta$.
We assume that $F_0$ and $F$ are Lipschitz-continuous mappings, i.e., that there exists a constant $L>0$ such that
\begin{equation} \label{eq:lips_F_0}
|F_0(x_1,\theta_1) - F_0(x_2,\theta_2)|_2 \leq
L\big(|x_1-x_2|_2 + |\theta_1-\theta_2|_2\big)
\end{equation}
and 
\begin{equation} \label{eq:lips_F}
\sup_{i=1,\ldots,k} |F_i(x_1,\theta_1) - F_i(x_2,\theta_2)|_2 \leq
L\big(|x_1-x_2|_2 + |\theta_1-\theta_2|_2\big)
\end{equation}
for every $(x_1,\theta_1),(x_2,\theta_2)\in \R^n\times\Theta$. 
In \eqref{eq:lips_F} we used $F_i(x,\theta)$ to denote the vector obtained by taking the $i^{\mathrm{th}}$ column of the matrix $F(x,\theta)$, for every $i=1,\ldots,k$. Similarly, for every $\theta\in\Theta$ we shall use $F^\theta_i:\R^n\to\R^n$ to denote the vector field corresponding to the $i^{\mathrm{th}}$ column of the matrix-valued application $F^\theta:\R^n\to\R^{n\times k}$.
We observe that \eqref{eq:lips_F_0}-\eqref{eq:lips_F} imply that the vector fields $F_0^\theta,F_1^\theta,\ldots, F_k^\theta$ are uniformly Lipschitz-continuous as $\theta$ varies in $\Theta$.
Another consequence of the Lipschitz-continuity conditions \eqref{eq:lips_F_0}-\eqref{eq:lips_F} is that the vector fields constituting the affine-control system \eqref{eq:ens_ctrl_sys} have sub-linear growth, uniformly with respect to the dependence on $\theta$. Namely, we have that there exists a constant $C>0$ such that 
\begin{equation}\label{eq:sub_lin_F_0}
\sup_{\theta\in\Theta} |F_0^\theta(x)|_2\leq 
C\big(|x|_2+1)
\end{equation}
and
\begin{equation}\label{eq:sub_lin_F}
\sup_{\theta\in\Theta}\, \sup_{i=1,\ldots,k}
|F^\theta_i(x)|_2\leq 
C\big(|x|_2+1)
\end{equation}
for every $x\in\R^n$.
Finally, let us consider the application $x_0:\Theta\to\R^n$ that prescribes the initial state of \eqref{eq:ens_ctrl_sys}, i.e., 
\begin{equation} \label{eq:def_x0}
x_0(\theta):= x_0^\theta
\end{equation}
for every $\theta \in \Theta$. We assume that $x_0$ is continuous. As a matter of fact, there exists a constant $C'>0$ such that
\begin{equation} \label{eq:bound_ens_init_val}
\sup_{\theta \in \Theta} |x_0(\theta)|_2 \leq C'.
\end{equation}
We set $\U:=L^2([0,1],\R^k)$ as the space of admissible controls, and we equip it with the usual Hilbert space structure given by the scalar product
\begin{equation} \label{eq:def_scal_prod}
\langle u,v\rangle_{L^2}
:= \int_0^1 \langle u(t),v(t) \rangle_{\R^k} \,dt
\end{equation}
for every $u,v\in\U$. 
For every $u\in\U$ and $\theta\in\Theta$, the curve $x_u^\theta:[0,1]\to\R^n$ denotes the solution of the Cauchy problem \eqref{eq:ens_ctrl_sys} corresponding to the system identified by $\theta$ and to the admissible control $u$.
We recall that, for every $u\in\U$ and $\theta\in\Theta$, the existence and uniqueness of the solution of \eqref{eq:ens_ctrl_sys} are guaranteed by the Carath\'eodory Theorem (see, e.g., \cite[Theorem~5.3]{H80}). 
Given $u\in\U$, we describe the evolution of the ensemble of control systems \eqref{eq:ens_ctrl_sys} through the mapping $X_u:[0,1]\times\Theta\to\R^n$ defined as follows:
\begin{equation} \label{eq:def_evol_ens}
X_u(t,\theta) := x_u^\theta(t)
\end{equation}
for every $(t,\theta)\in[0,1]\times\Theta$. 
In other words, for every $u\in\U$ the application $X_u$ collects the trajectories of the ensemble of control systems \eqref{eq:ens_ctrl_sys}.
We study the properties of the mapping $X_u$ in Subsection~\ref{subsec:mapp_Xu} below.  
Before proceeding, we recall some elementary facts in functional analysis.

\begin{subsection}{General results in functional 
analysis} \label{subsec:fun_an_recall}
We begin by recalling some basic facts about the space of admissible controls $\U:=L^2([0,1],\R^k)$.
First of all, the linear inclusion $\U\hookrightarrow L^1([0,1],\R^k)$ is continuous, and from \eqref{eq:norm_eucl_equiv} and the Jensen inequality it follows that 
\begin{equation} \label{eq:norm_ineq}
||u||_{L^1} := \int_0^1 |u(\tau)|_1 \,d\tau
\leq  \sqrt k ||u||_{L^2}
\end{equation}
for every $u\in\U$. 
We shall often make use of $L^2$-weakly convergent sequences. Given a sequence $(u_m)_{m\in\NN}\subset \U$, we say that $(u_m)_{m\in\NN}$ is convergent to $u\in\U$ with respect to the weak topology of $L^2$ if 
\[
\lim_{m\to\infty}\langle v,u_m\rangle_{L^2} =
\langle v,u\rangle_{L^2}
\]
for every $v\in\U$, and we write $u_m\weak_{L^2} u$ as $m\to\infty$.
If $u_m\weak_{L^2}u$ as $m\to\infty$, then we have  
\begin{equation} \label{eq:norm_weak_semicont}
||u||_{L^2} \leq 
\liminf_{m\to\infty}||u_m||_{L^2}.
\end{equation}
Finally, we recall that any bounded sequence $(u_m)_{m\in\NN}$ is pre-compact with respect to the $L^2$-weak topology. For further details on weak topologies of Banach spaces, the reader is referred to \cite[Chapter~3]{B11}.
We conclude this part with the following fact concerning the one-dimensional Sobolev space $H^1([a,b],\R^d):= W^{1,2}([a,b],\R^d)$. For a complete survey on the topic, we recommend \cite[Chapter~8]{B11}.

\begin{proposition} \label{prop:sob_Holder}
Let $f:[a,b]\to\R^d$ be a function in
$H^1([a,b],\R^d)$. Then, $f$ is
H\"older-continuous with exponent $\frac12$,
namely
\begin{equation*}
|f(t_1)-f(t_2)|_2 \leq ||f'||_{L^2}|t_1-t_2|^{\frac12} 
\end{equation*}
for every $t_1,t_2\in[a,b]$, where 
$f'\in L^2([0,1],\R^d)$ denotes the weak 
derivative of $f$. 
\end{proposition}
%\begin{proof}
%The fact that
%$f\in H^1([a,b],\R^d)$
%implies that it is absolutely continuous
%(see, e.g., \cite[Theorem~8.2]{B11}). Thus,
%using the Cauchy-Schwartz inequality,
%we deduce that
%\begin{equation*}
%|f(t_1)-f(t_2)|_2\leq \int_{[t_1,t_2]}|f'(\tau)|_2
%\, d\tau \leq 
%\left( \int_{[t_1,t_2]}|f'(\tau)|_2^2 \, d\tau
%\right)^{\frac12} |t_1-t_2|^{\frac12}
%\end{equation*}
%for every $t_1,t_2\in[a,b]$, and 
%this implies the thesis.
%\end{proof}
\end{subsection}

\begin{subsection}{Trajectories of the controlled
ensemble}  \label{subsec:mapp_Xu}

We now investigate the evolution of the ensemble
of control systems \eqref{eq:ens_ctrl_sys}
when we consider a sequence of $L^2$-weakly
convergent admissible controls.
The proof is postponed to the end of the present
subsection.

\begin{proposition} \label{prop:unif_conv_map_X}
Let us consider a sequence of admissible controls
$(u_m)_{m\in\NN} \subset \U$
such that $u_m\weak_{L^2}u_\infty$ as
$m\to\infty$.
For every $m\in \NN \cup \{ \infty\}$,
let $X_m:[0,1]\times \Theta\to \R^n$
be the application defined in
\eqref{eq:def_evol_ens} that collects the
trajectories of the ensemble of control
systems \eqref{eq:ens_ctrl_sys} 
corresponding to the admissible control $u_m$.
Therefore, we have that
\begin{equation} \label{eq:unif_conv_map_X}
\lim_{m\to\infty}\, \sup_{(t,\theta)\in [0,1]\times\Theta} |X_m(t,\theta) - X_\infty(t,\theta)|_2
=0. 
\end{equation}
\end{proposition}

\begin{remark}\label{rmk:aff_ctrl_C0_conv}
Proposition~\ref{prop:unif_conv_map_X} is the cornerstone of the theoretical results presented in this paper. Indeed, the fact fact that the trajectories of the ensemble \eqref{eq:ens_ctrl_sys} are uniformly convergent when the corresponding controls are $L^2$-weakly convergent is used both to prove the existence of optimal controls (see Theorem~\ref{thm:exist_minim}) and to establish the $\Gamma$-convergence result (see Theorem~\ref{thm:G_conv}). We stress that the fact that the systems in the ensemble \eqref{eq:ens_fields} have affine dependence in the controls is crucial for the proof of Proposition~\ref{prop:unif_conv_map_X}.
\end{remark}

In view of the next auxiliary result, we
introduce some notations. 
For every $\theta\in \Theta$, we define 
$\tilde F^\theta:\R^n\to\R^{n\times(k+1)}$
as follows:
\begin{equation}\label{eq:aux_lin_ctrl}
\tilde F^\theta(x):= (F_0^\theta(x), F^\theta(x)),
\end{equation}
for every $x\in\R^n$, i.e., 
we add the column $F^\theta_0(x)$ to the
$n\times k$ matrix $F^\theta(x)$.
Similarly, for every $u\in\U=
L^2([0,1],\R^{k})$, we consider 
the extended control 
$\tilde{u}\in \tilde \U:= L^2([0,1],\R^{k+1})$
defined as
\begin{equation}\label{eq:aux_ext_ctrl}
\tilde u(t) = (1,u(t))^T
\end{equation}
for every $t\in[0,1]$, i.e., we add the 
component $u_0= 1$ to the column-vector
$u(t)$.

\begin{lemma} \label{lem:point_conv_X_m}
Let us consider a sequence of admissible controls
$(u_m)_{m\in\NN} \subset \U$
such that $u_m\weak_{L^2}u_\infty$ as
$m\to\infty$.
For every $m\in \NN \cup \{ \infty\}$
and for every $\theta \in \Theta$, let
$x_m^\theta:[0,1]\to\R^n$ be the solution
of \eqref{eq:ens_ctrl_sys} corresponding to the
ensemble parameter $\theta$ and to the admissible
control $u_m$.
Then, for every
$\theta \in \Theta$ we have
\begin{equation}\label{eq:point_conv_X_m}
\lim_{m\to\infty}||x_m^\theta -x_\infty^\theta||_{C^0}=0.
\end{equation}
\end{lemma}
\begin{proof}
Let us fix $\theta\in\Theta$. 
By means of the 
matrix-valued function $\tilde F:\R^n\to
\R^{n\times(k+1)}$ and the extended control
$\tilde u:[0,1]\to\R^{k+1}$ defined in
\eqref{eq:aux_lin_ctrl} and \eqref{eq:aux_ext_ctrl}
respectively, we can equivalently rewrite 
the affine-control system \eqref{eq:ens_ctrl_sys}
corresponding to $\theta$ as follows:
\begin{equation} \label{eq:lin_ctrl_rewr}
\begin{cases}
\dot x^\theta 
= \tilde F^\theta(x^\theta) \tilde u,\\
\xth(0) = \xth_0,
\end{cases}
\end{equation}
for every $u\in \U$. In other words, any 
solution $x_u^\theta:[0,1]\to\R^n$
of \eqref{eq:ens_ctrl_sys} 
corresponding to
the admissible control $u\in\U$ is in turn
a solution of the linear-control system
\eqref{eq:lin_ctrl_rewr} corresponding
to the extended control $\tilde u
\in \tilde \U$. %L^2([0,1],\R^{k+1})$
On the other hand, 
the convergence $u_m\weak_{L^2}u_\infty$
as $m\to\infty$
implies the convergence of the respective extended
controls, i.e., $\tilde u_m\weak_{L^2}\tilde 
u_\infty$ as $m\to\infty$.
Therefore, $(x_m^\theta)_{m\in\NN}$ is the 
sequence of solutions of the linear-control system
\eqref{eq:lin_ctrl_rewr} 
corresponding to the $L^2$-weakly convergent 
sequence of controls $(\tilde u_m)_{m\in\NN}$.
Moreover, $x^\theta_\infty$ is the solution of
\eqref{eq:lin_ctrl_rewr} associated with the
weak-limiting control $\tilde u_\infty$.
Using \cite[Lemma~7.1]{S1}, we deduce 
 \eqref{eq:point_conv_X_m}.
\end{proof}

We are now in position to prove Proposition~\ref{prop:unif_conv_map_X}.

\begin{proof}[Proof of Proposition~\ref{prop:unif_conv_map_X}]
Let us consider a $L^2$-weakly convergent sequence 
$(u_m)_{m\in\NN} \subset \U$
such that $u_m\weak_{L^2}u_\infty$ as
$m\to\infty$. 
We immediately deduce that there exists
$R>0$ such that
$||u_m||_{L^2}^2 \leq R$ for every $m\in \NN\cup
\{  \infty\}$. 
%\begin{equation*}
%||u_m||_{L^2}\leq R, \quad \forall m\in \NN\cup
%\{ \infty \}.
%\end{equation*}
Thus, in virtue of 
Lemma~\ref{lem:unif_cont_X}, we 
deduce that the sequence of mappings
$\{ X_m:[0,1]\times \Theta \to\R^n \}_{m\in \NN}$
is uniformly equi-continuous, while
Lemma~\ref{lem:bound_X} guarantees that it
is uniformly equi-bounded. Therefore, 
applying the Ascoli-Arzel\`a Theorem (see, e.g.,
\cite[Theorem~4.25]{B11}), we deduce that
the family 
$( X_m )_{m\in \NN}$
is pre-compact with respect to the 
strong topology of the Banach space
$C^0([0,1]\times\Theta,\R^n)$.
Finally, Lemma~\ref{lem:point_conv_X_m}
implies that
\begin{equation*}
\lim_{m\to\infty} X_m(t,\theta) =
X_\infty (t,\theta) 
\end{equation*}
for every $(t,\theta)\in[0,1]\times\Theta$. 
In particular, we deduce that the 
set of limiting points of the pre-compact sequence 
$( X_m )_{m\in \NN}$ is reduced to
the single-element set
$\{ X_\infty \}$. 
This proves \eqref{eq:unif_conv_map_X}.
\end{proof}
 
\end{subsection}

\begin{subsection}{Adjoint variables of the controlled ensemble} \label{subsec:Lambda}
In this subsection we introduce a function $\Lambda_u$, which will play a crucial role in Section~\ref{sec:PMP_ens}. 
Here we consider an assigned function $a:\R^n\times \Theta\to\R$ such that $(x,\theta)\mapsto \nabla_x a(x,\theta)$ is continuous. Moreover, we further require that $(x,\theta)\mapsto \frac{\partial}{\partial x}F_i(x,\theta)$ is continuous for every $i=0,\ldots,k$.
For every $u\in \U$ and every $\theta\in \Theta$, we define the function $\lambda_u^\theta:[0,1]\to (\R^n)^*$ as the solution of the following differential equation
\begin{equation} \label{eq:def_lambda_theta_u}
\begin{cases}
\dot \lambda_u^\theta(t) = - \lambda_u^\theta(t)
\left( 
\frac{\partial F_0(x_u^\theta(t), \theta)}{\partial x}
+ \sum_{i=1}^ku_i(t) 
\frac{\partial F_i(x_u^\theta(t), \theta)}{\partial x} 
\right),\\
\lambda_u^\theta(1) = \nabla_x a(x_u^\theta(1),\theta),
\end{cases}
\end{equation}
where the curve $x_u^\theta:[0,1]\to\R^n$ is the solution of the Cauchy problem \eqref{eq:ens_ctrl_sys} corresponding to the system identified by $\theta$ and to the admissible control $u$. 
We insist on the fact that in this paper $\lambda_u^\theta$ is always understood as a row-vector, as well as any other element of $(\R^n)^*$. The existence and the uniqueness of the solution of \eqref{eq:def_lambda_theta_u} follow as a standard application of the Carath\'eodory Theorem (see, e.g., \cite[Theorem~5.3]{H80}).
Similarly as done in the previous subsection, for every $u\in\U$ we introduce the function $\Lambda_u:[0,1]\times \Theta\to (\R^n)^*$ defined as 
\begin{equation} \label{eq:def_Lambda}
\Lambda_u(t,\theta):= \lambda_u^\theta(t).
\end{equation}
In the case of a sequence of weakly convergent controls $(u_m)_{m\in\NN}$, for the corresponding sequence $(\Lambda_{u_m})_{m\in\NN}$ we can establish a result analogue to Proposition~\ref{prop:unif_conv_map_X}.
 
\begin{proposition} \label{prop:unif_conv_map_Lambda}
Let us assume that the mappings 
$(x,\theta)\mapsto 
\frac{\partial}{\partial x}F_i(x,\theta)$ 
are continuous for every $i=0,\ldots,k$, as well as
the gradient $(x,\theta)\mapsto \nabla_x a(x,\theta)$.
Let us consider a sequence of admissible controls
$(u_m)_{m\in\NN} \subset \U$
such that $u_m\weak_{L^2}u_\infty$ as
$m\to\infty$.
For every $m\in \NN \cup \{ \infty\}$,
let $\Lambda_m:[0,1]\times \Theta\to (\R^n)^*$
be the application defined in
\eqref{eq:def_Lambda} that collects the
adjoint variables \eqref{eq:def_lambda_theta_u}
corresponding to the admissible control $u_m$.
Then, we have that
\begin{equation} \label{eq:unif_conv_map_Lambda}
\lim_{m\to\infty}\, \sup_{(t,\theta)\in [0,1]\times\Theta} |\Lambda_m(t,\theta) - \Lambda_\infty(t,\theta)|_2
=0. 
\end{equation}
\end{proposition}

Before detailing the proof of Proposition~\ref{prop:unif_conv_map_Lambda}, we establish an auxiliary result with a similar flavor as Lemma~\ref{lem:point_conv_X_m}.

\begin{lemma} \label{lem:point_conv_Lambda}
Let us assume that the mappings 
$(x,\theta)\mapsto 
\frac{\partial}{\partial x}F_i(x,\theta)$ 
are continuous for every $i=0,\ldots,k$, as well as
the gradient $(x,\theta)\mapsto \nabla_x a(x,\theta)$.
Let us consider a sequence of admissible controls
$(u_m)_{m\in\NN} \subset \U$
such that $u_m\weak_{L^2}u_\infty$ as
$m\to\infty$.
For every $m\in \NN \cup \{ \infty\}$
and for every $\theta \in \Theta$, let
$\lambda_m^\theta:[0,1]\to(\R^n)^*$ be the solution
of \eqref{eq:def_lambda_theta_u} corresponding to the
ensemble parameter $\theta$ and to the admissible
control $u_m$.
Then, for every $t\in [0,1]$ and for every
$\theta \in \Theta$, we have
\begin{equation}\label{eq:point_conv_Lmd}
\lim_{m\to\infty}||
\lambda_m^\theta(t)-\lambda_\infty^\theta(t)||_{C^0}=0.
\end{equation}
\end{lemma}
\begin{proof}
The weak convergence $u_m\weak_{L^2}u_\infty$ as
$m\to\infty$ implies that there exists $R>0$
such that $||u_m||_{L^2}\leq R$ for every
$m\in \NN \cup \{ \infty  \}$.
Let us fix $\theta \in \Theta$.
With the same argument as in the proof of 
Lemma~\ref{lem:unif_holder_time_Lambda} we deduce
that the sequence 
$(\lambda_m^\theta)_{m\in \NN}\subset H^1([0,1],\R^k)$
is equi-bounded. Therefore, there exists a 
weakly convergent subsequence 
$(\lambda_{m_\ell}^\theta)_{\ell\in \NN}$ such that
$\lambda_{m_\ell}^\theta \weak_{H^1}\bar \lambda^\theta$ as $\ell\to\infty$. Moreover, this implies that
$\dot \lambda_{m_\ell}^\theta \weak_{L^2}\dot{\bar \lambda}^\theta$ as $\ell\to\infty$, while from the
compact inclusion $H^1 \hookrightarrow C^0$
we deduce that $\lambda_{m_\ell}^\theta \to_{C^0}\bar \lambda^\theta$ as $\ell\to\infty$.
In particular, this last convergence
and Lemma~\ref{lem:point_conv_X_m} imply that
\begin{equation} \label{eq:final_datum_lem}
\bar \lambda^\theta(1) = \lim_{\ell \to \infty}
\lambda_{m_\ell}^\theta(1) = 
\lim_{\ell \to \infty} \nabla_x a(x^\theta_{m_\ell}(1),\theta) = \nabla_x a(x^\theta_\infty(1),\theta),
\end{equation}
where for every $m\in \NN\cup\{ \infty \}$ the curve
$x^\theta_m:[0,1]\to\R^n$ denotes the solution
of \eqref{eq:ens_ctrl_sys} corresponding to the
control $u_m$ and to the parameter $\theta$.
We want to prove that $\bar \lambda^\theta:[0,1]\to
(\R^n)^*$ is the
solution of \eqref{eq:def_lambda_theta_u}
corresponding to the control $u_\infty$.
We recall that 
\begin{equation} \label{eq:diff_eq_lambda_m}
\dot \lambda_{m_\ell}^\theta =
\lambda_{m_\ell}^\theta
 \left( 
\frac{\partial F_0(x_{m_\ell}^\theta(\cdot), \theta)}{\partial x}
+ \sum_{i=1}^ku_{i, {m_\ell}} 
\frac{\partial F_i(x_{m_\ell}^\theta(\cdot), \theta)}{\partial x} 
\right)
\end{equation}
for every $\ell\in \NN$.
We observe that, in virtue of Lemma~\ref{lem:bound_X},
there exists $K_R\subset \R^n$ such that
$x_m^\theta(t)\in K_R$ for every $m\in \NN\cup\{\infty\}$ and for every $(t,\theta)\in [0,1]\times \Theta$.
Then, owing to the continuity of 
the mappings 
$(x,\theta)\mapsto 
\frac{\partial}{\partial x}F_i(x,\theta)$ 
for every $i=0,\ldots,k$, we deduce 
the convergence 
$\frac{\partial}{\partial x}F_i(x_{m_\ell}(\cdot),\theta)\to_{C^0} 
\frac{\partial}{\partial x}F_i(x_{\infty}(\cdot),\theta)$ as $\ell\to\infty$ for every $i=0,\ldots,k$.
Summarizing, we have that
\begin{equation} \label{eq:resume_conv}
\begin{cases}
\dot \lambda_{m_\ell}^\theta \weak_{L^2}\dot{\bar \lambda}^\theta\\
\lambda_{m_\ell}^\theta \to_{C^0}\bar \lambda^\theta\\
u_{i, m_\ell} \weak_{L^2} u_{i,\infty} & \mbox{for every } i=1,\ldots,k,\\ 
\frac{\partial}{\partial x}F_i(x_{m_\ell}(\cdot),\theta)\to_{C^0} 
\frac{\partial}{\partial x}F_i(x_{\infty}(\cdot),\theta) & \mbox{for every } i=0,\ldots,k,
\end{cases} \quad
\mbox{ as } \ell\to\infty.
\end{equation}
Combining \eqref{eq:diff_eq_lambda_m} and
\eqref{eq:resume_conv}, we derive that 
\begin{equation}\label{eq:diff_eq_lambda_inf}
\dot {\bar \lambda}^\theta =
\bar \lambda^\theta
 \left( 
\frac{\partial F_0(x_{\infty}^\theta(\cdot), \theta)}{\partial x}
+ \sum_{i=1}^ku_{i, \infty} 
\frac{\partial F_i(x_{\infty}^\theta(\cdot), \theta)}{\partial x} 
\right).
\end{equation}
The identities \eqref{eq:final_datum_lem} and
\eqref{eq:diff_eq_lambda_inf} show that
$\bar \lambda^\theta
\equiv \lambda_\infty^\theta$, where  
$\lambda_\infty^\theta:[0,1]\to
(\R^n)^*$ is the unique
solution of \eqref{eq:def_lambda_theta_u}
corresponding to the control $u_\infty$.
Hence, since
\textit{any} $H^1$-weakly convergent subsequence
of $(\lambda_m^\theta)_{m\in\NN}$
must converge to $\lambda_\infty^\theta$,
we get \eqref{eq:point_conv_Lmd}.
Since this argument holds for every choice of $\theta\in \Theta$, we deduce the thesis.
\end{proof}

We are now able to prove 
Proposition~\ref{prop:unif_conv_map_Lambda}.

\begin{proof}[Proof of Proposition~\ref{prop:unif_conv_map_Lambda}]
The argument is the same as in the proof of
Proposition~\ref{prop:unif_conv_map_X}.
Namely, Lemma~\ref{lem:point_conv_Lambda} guarantees
the pointwise convergence of the mappings
$(\Lambda_m)_{m\in\NN}$ to $\Lambda_\infty$,
while Lemma~\ref{lem:bound_Lambda}
and Lemma~\ref{lem:unif_cont_Lambda} ensure, respectively,
that
the elements of the sequence are uniformly equi-bounded
and uniformly equi-continuous. 
\end{proof}

\end{subsection}

\begin{subsection}{Gradient field for
affine-control systems with
end-point cost} \label{subsec:aff_grad_field}
In this subsection we generalize to the case of affine-control systems some of the results obtained in \cite{S1} in the framework of linear-control systems with end-point cost. As we shall see, the strategy that we pursue consists in embedding the affine-control system into a larger linear-control system, similarly as done in the proof of Lemma~\ref{lem:point_conv_X_m}. Therefore, we can exploit a consistent part of the machinery developed in \cite{S1} to cover the present case.
Let us consider a {\it single} affine-control system on $\R^n$ of the form 
\begin{equation} \label{eq:aff_sys}
\begin{cases}
\dot x(t) = F_0(x(t)) + F(x(t))u(t), &
\mbox{for a.e. }t\in[0,1],\\ 
x(0) = x_0,
\end{cases}
\end{equation}
where $F_0:\R^n\to\R^n$ and $F:\R^n\to\R^{n\times k}$ are $C^2$-regular applications that design the affine-control system, and $u\in \U=L^2([0,1],\R^k)$ is the control. 
We introduce the functional $\J:\U\to\R$ defined on the space of admissible controls as follows:
\begin{equation} \label{eq:def_J}
\J(u):= a(x_u(1)) + \frac{\beta}{2}||u||_{L^2}^2
\end{equation}
for every $u\in\U$, where $a:\R^n\to\R$ is a $C^2$-regular function and $\beta>0$ a positive parameter. 
After proving that the functional $\J$ is differentiable, we provide the Riesz's representation of the differential $d_u\J:\U\to\R$. 

Before proceeding, it is convenient to introduce the linear-control system in which we embed \eqref{eq:aff_sys}. Similar to \eqref{eq:aux_lin_ctrl}, let $\tilde F:\R^n\to\R^{n\times(k+1)}$ be the function defined as 
\begin{equation} \label{eq:ext_lin_matrix}
\tilde F(x) := (F_0(x),F(x))
\end{equation}
for every $x\in\R^n$.
If we define the extended space of admissible controls as $\tilde\U:=L^2([0,1],\R^{k+1})$, we may consider the following linear-control system
\begin{equation} \label{eq:lin_ctrl_ext}
\begin{cases}
\dot x(t) = \tilde F(x(t)) \tilde u(t)&
\mbox{for a.e. } t\in[0,1],\\
x(0)=x_0,
\end{cases}
\end{equation}
where $\tilde u\in \tilde \U$. 
We observe that we can recover the affine system \eqref{eq:aff_sys} by restricting the set of admissible controls in \eqref{eq:lin_ctrl_ext} to the image of the affine embedding $i:\U\to \tilde \U$ defined as 
\begin{equation} \label{eq:def_embed_ctrls}
i[u] := \left(\begin{matrix}
1 \\
u
\end{matrix}\right).
\end{equation}
We introduce the extended cost functional $\tilde\J:\tilde \U\to\R$ as 
\begin{equation}\label{eq:def_tilde_J}
\tilde \J(\tilde u) := a(x_{\tilde u}(1)) +
\frac\beta2 ||\tilde u||_{L^2}^2
\end{equation}
for every $\tilde u\in\tilde \U$,
where $x_{\tilde u}:[0,1]\to\R^n$ is the absolutely continuous solution of \eqref{eq:lin_ctrl_ext} corresponding to the control $\tilde u$.
To avoid confusion, in the present subsection we denote by $\langle\cdot,\cdot\rangle_{\U}$ and $\langle\cdot,\cdot\rangle_{\tilde \U}$ the scalar products in $\U$ and $\tilde \U$, respectively.
In the next result we prove that the functional $\J:\U\to\R$ defined in \eqref{eq:def_J} is differentiable.

\begin{proposition} \label{prop:G_diff_cost}
Let us assume that $F_0:\R^n\to\R^n$ 
and $F:\R^n\to\R^{n\times k}$ are 
$C^1$-regular, as well as the 
function $a:\R^n\to\R$ designing the end-point
cost.
Then, the functionals $\J:\U\to\R$ 
and $\tilde \J:\tilde \U\to\R$ 
defined, respectively, in \eqref{eq:def_J}
and in \eqref{eq:def_tilde_J} are
Gateaux differentiable at every 
point of their respective domains.
\end{proposition}

\begin{proof}
We observe that the functional $\J:\U\to\R$ 
satisfies the following identity: 
\begin{equation}\label{eq:ident_ext_fun}
\J(u) = \tilde \J ( i(u)) -\frac\be2
\end{equation}
for every $u\in\U$, where $i:\U\to\tilde \U$ is the affine embedding reported in \eqref{eq:def_embed_ctrls}.
Since $i:\U\to\tilde \U$ is analytic, the proof reduces to showing that the functional $\tilde\J:\tilde\U\to\R$ is Gateaux differentiable. 
This is actually the case, since $\tilde u\mapsto \frac\be2 ||\tilde u||_{L^2}$ is smooth, while the first term at the right-hand side of \eqref{eq:def_tilde_J} (i.e., the end-point cost)
is Gateaux differentiable owing to \cite[Lemma~3.1]{S1}.
\end{proof}

By differentiation of the identity \eqref{eq:ident_ext_fun}, we deduce that
\begin{equation}\label{eq:ident_diff_ext_fun}
d_u\J(v) = d_{i[u]}\tilde\J \big( i_\#[v] \big)
\end{equation}
for every $u,v\in\U$, where we have introduced the linear inclusion $i_\#:\U\to\tilde\U$ defined as  
\begin{equation} \label{eq:def_i_sharp}
i_\#[v]:= 
\left(\begin{matrix}
0 \\
v
\end{matrix}\right)
\end{equation}
for every $v\in\U$.
In virtue of Proposition~\ref{prop:G_diff_cost}, we can consider the vector field $\G:\U\to\U$ that represents the differential of the functional $\J:\U\to\R$. 
Namely, for every $u\in \U$, let $\G[u]$ be the unique element of $\U$ such that
\begin{equation}\label{eq:def_G}
\langle\G[u],v\rangle_{\U} = d_u\J(v)
\end{equation} 
for every $v\in\U$.
Similarly, let us denote by $\tilde \G:\tilde\U \to \tilde \U$  the vector field such that
\begin{equation}\label{eq:def_tilde_G}
\langle\tilde \G[\tilde u],\tilde v\rangle_{\tilde 
\U} = d_{\tilde u} \tilde \J(\tilde v)
\end{equation}
for every $\tilde u,\tilde v\in \tilde\U$.
In \cite{S1} it was derived the expression of the vector field $\tilde\G$ associated with the linear-control system \eqref{eq:lin_ctrl_ext} and to the cost \eqref{eq:def_tilde_J}.
In the next result we use it in order to obtain the expression of $\G$. We use the notation $F(x)^T$ to denote the matrix in $\R^{k\times n}$ obtained by the transposition of the matrix $F(x)\in \R^{n\times k}$, for every $x\in \R^n$. The analogue convention holds for $\tilde F(x)^T$, for every $x\in\R^n$.  

\begin{theorem}\label{thm:grad_field_aff}
Let us assume that $F_0:\R^n\to\R^n$ 
and $F:\R^n\to\R^{n\times k}$ are 
$C^1$-regular, as well as the 
function $a:\R^n\to\R$ designing the end-point
cost. Let $\G:\U\to\U$ be the 
gradient vector field on $\U$
that satisfies \eqref{eq:def_G}.
Then, for every $u\in\U$ we have
\begin{equation}\label{eq:grad_field_aff}
\G[u](t) = F(x_u(t))^T\lambda_u^T(t) + \be u(t)
\end{equation}
for a.e. $t\in[0,1]$, where
$x_u:[0,1]\to\R^n$ is the solution of 
\eqref{eq:aff_sys} corresponding to the control
$u$, and $\lambda_u:[0,1]\to(\R^n)^*$ is the
absolutely continuous curve of covectors that
solves
\begin{equation}\label{eq:adj_eq_aff}
\begin{cases}
\dot \lambda_u(t) = -\lambda_u(t)
\left(\frac{\partial F_0(x_u(t))}{\partial x} + 
\sum_{i=1}^k u_i(t)
\frac{\partial F_i(x_u(t))}{\partial x}
\right) &\mbox{a.e. in } [0,1],\\
\lambda_u(1) = \nabla a(x_u(1)).
\end{cases}
\end{equation}
\end{theorem}     

\begin{remark}\label{rmk:row_vec_1}
In this paper, we understand the elements of $(\R^n)^*$ as row-vectors. Therefore, for every $t\in[0,1]$, $\lambda_u(t)$ should be read as a row-vector. This should be considered to give meaning to \eqref{eq:adj_eq_aff}.
\end{remark}

\begin{proof}[Proof of Theorem~\ref{thm:grad_field_aff}]
In virtue of \eqref{eq:ident_diff_ext_fun},
from the definitions \eqref{eq:def_G} and
\eqref{eq:def_tilde_G} we deduce that
\begin{equation} \label{eq:ident_G_tilde_G_aux}
\langle\G[u],v\rangle_{\U} =
\langle\tilde \G[i[u]],i_\# [v]\rangle_{\tilde\U}
= \langle\pi \tilde \G[i[u]],v\rangle_{\U}
\end{equation}
for every $u,v\in\U$, where
$\tilde \G:\tilde \U\to \tilde \U$
is the gradient vector field corresponding to
the functional $\tilde \J:\tilde\U\to\R$, and
 $\pi:\tilde \U\to\U$ 
is the linear application
\begin{equation} \label{eq:def_pi}
\pi: 
\left(\begin{matrix}
\tilde v_0 \\
\vdots\\
\tilde v_k
\end{matrix}\right)
\mapsto
\left(\begin{matrix}
\tilde v_1 \\
\vdots\\
\tilde v_k
\end{matrix}\right)
\end{equation}
for every $\tilde v\in \tilde \U$.
Therefore, we can rewrite 
\eqref{eq:ident_G_tilde_G_aux} as
\begin{equation}\label{eq:ident_G_tilde_G}
\G = \pi \circ \tilde \G\circ i,
\end{equation} 
where $i$ and $\pi$ are defined, respectively, in
\eqref{eq:def_embed_ctrls} and in
\eqref{eq:def_pi}.
This implies that we can deduce the expression
of $\G$ from the one of $\tilde \G$. 
In particular, from \cite[Remark~8]{S1}
it follows that for every $\tilde u\in\tilde \U$
we have
\begin{equation}\label{eq:grad_field_lin}
\tilde \G[\tilde u](t) =
\tilde F(x_{\tilde u}(t))^T\lambda_{\tilde u}^T(t)
+ \beta \tilde u(t)
\end{equation}
for a.e. $t\in[0,1]$, where
$x_{\tilde u}:[0,1]\to\R^n$ is the solution of 
\eqref{eq:lin_ctrl_ext} corresponding to the control
$\tilde u$, and $\lambda_{\tilde u}:[0,1]\to(\R^n)^*$ is the
absolutely continuous curve of covectors that
solves
\begin{equation}\label{eq:adj_eq_lin}
\begin{cases}
\dot \lambda_{\tilde u}(t) = -\lambda_{\tilde u}(t)
\sum_{i=0}^k \left( \tilde u_i(t)
\frac{\partial \tilde F_i(x_{\tilde u}(t))}{\partial x}
\right) &\mbox{for a.e. }t\in[0,1],\\
\lambda_{\tilde u}(1) = \nabla a(x_{\tilde u}(1)).
\end{cases}
\end{equation}
We stress the fact that the summation index
in \eqref{eq:adj_eq_lin} starts from $0$.
Then, the thesis follows immediately from 
\eqref{eq:ident_G_tilde_G}-\eqref{eq:adj_eq_lin}.
\end{proof}

\begin{remark} \label{rmk:reg_grad}
The identity \eqref{eq:ident_G_tilde_G} implies that the gradient field $\G:\U\to\U$ is at least as regular as $\tilde\G:\tilde\U\to\tilde\U$. In particular, under the further assumption that $F_0:\R^n\to\R^n$, $F:\R^n\to\R^{n\times k}$ and $a:\R^n\to\R$ are $C^2$-regular, from \cite[Lemma~3.2]{S1} it follows that $\tilde\G:\tilde\U\to\tilde\U$ is Lipschitz-continuous on the bounded sets of $\tilde \U$. In particular, under the same regularity hypotheses, $\G:\U\to\U$ is Lipschitz-continuous on the bounded sets of $\U$.  
\end{remark}
\end{subsection}

\end{section}

\begin{section}{Optimal control of ensembles} \label{sec:opt_ctrl_ens}
In this section we formulate a minimization problem for the ensemble of affine-control systems \eqref{eq:ens_ctrl_sys}. Namely, let us consider a non-negative continuous mapping $a:[0,1]\times\R^n\times\Theta\to\R_+$, a positive real number $\beta>0$ and a Borel probability measure $\nu$ on the time interval $[0,1]$. Therefore, for every $\theta\in\Theta$ we can study the following optimal control problem: 
\begin{equation} \label{eq:theta_fixed_opt_ctrl}
\int_0^1 a(t,x_u^\theta(t),\theta)\,d\nu(t)+
\frac\beta2 ||u||_{L^2}^2\to\min,
\end{equation}
where the curve $x_u^\theta:[0,1]\to\R^n$ is the solution of \eqref{eq:ens_ctrl_sys} corresponding to the parameter $\theta\in\Theta$ and to the admissible control $u\in\U$.
We recall that the ensemble of control systems \eqref{eq:ens_ctrl_sys} is aimed at modeling our partial knowledge of the data of the controlled dynamical system. Therefore, it is natural to assume that the space of parameters $\Theta$ is endowed with a Borel probability measure $\mu$ that quantifies this uncertainty. 
In view of this fact, we can formulate an optimal control problem for the ensemble of control systems \eqref{eq:ens_ctrl_sys} as follows:
\begin{equation} \label{eq:opt_ctrl_ens}
\int_\Theta 
\int_0^1a(t,x_u^\theta(t),\theta)\,d\nu(t)
\,d\mu(\theta)+
\frac\beta2 ||u||_{L^2}^2\to\min.
\end{equation}
The minimization problem \eqref{eq:opt_ctrl_ens} is obtained by averaging out the parameters $\theta \in \Theta$ in the optimal control problem \eqref{eq:theta_fixed_opt_ctrl} through the probability measure $\mu$.

In this section we study the variational problem \eqref{eq:opt_ctrl_ens}, and we prove that it admits a solution. 
Before proceeding, we introduce the functional $\F:\U\to\R_+$ associated with the minimization problem \eqref{eq:opt_ctrl_ens}. For every admissible control $u\in\U$, we set 
\begin{equation} \label{eq:def_fun_ens}
\F(u):= \int_\Theta 
\int_0^1a(t,x_u^\theta(t),\theta)\,d\nu(t)
\,d\mu(\theta)+
\frac\beta2 ||u||_{L^2}^2.
\end{equation}
We first prove an auxiliary lemma regarding the integral cost in \eqref{eq:opt_ctrl_ens}.

\begin{lemma} \label{lem:conv_int_term}
Let us consider a sequence of admissible controls
$(u_m)_{m\in\NN} \subset \U$
such that $u_m\weak_{L^2}u_\infty$ as
$m\to\infty$. For every $m\in \NN \cup \{ \infty\}$,
let $Y_m:[0,1]\times \Theta\to \R_+$
be defined as follows:
\begin{equation}\label{eq:def_Y_m}
Y_m(t,\theta):= a(t,X_m(t,\theta),\theta),
\end{equation}
where $X_m:[0,1]\times \Theta\to \R^n$
is the application defined in
\eqref{eq:def_evol_ens} 
corresponding to the admissible control $u_m$. 
Then, we have that
\begin{equation} \label{eq:conv_int_term}
\lim_{m\to\infty}\, \sup_{(t,\theta)\in[0,1]\times\Theta}|Y_m(t,\theta) - Y_\infty(t,\theta)|
=0.
\end{equation}
\end{lemma}
\begin{proof}
Since the sequence $(u_m)_{m\in\NN}$ is 
weakly convergent, there exists
$R>0$ such that $||u_m||_{L^2}\leq R$
for every $m\in\NN\cup\{ \infty \}$.
For every $m\in \NN\cup\{ \infty \}$, let 
$X_m:[0,1]\times\Theta \to\R^n$ be the 
application defined in \eqref{eq:def_evol_ens}
corresponding to the control $u_m$.
In virtue of Lemma~\ref{lem:bound_X}, there exists a compact set 
$K\subset \R^n$ such that 
\begin{equation*}
X_m(t,\theta) \in K
\end{equation*}
for every $(t,\theta)\in[0,1]\times\Theta$
and for every $m\in\NN\cup\{ \infty \}$.
Recalling that the function
$a:[0,1]\times\R^n\times\Theta\to\R_+$
that defines the integral term in
\eqref{eq:def_fun_ens} is assumed
to be continuous, it follows that the
restriction
\begin{equation*}
\tilde a := a|_{[0,1]\times K\times\Theta}
\end{equation*}
is uniformly continuous.
In addition, Proposition~\ref{prop:unif_conv_map_X}
guarantees that 
$X_m\to_{C^0}X_\infty$ as 
$m\to\infty$. 
Therefore, observing that
\begin{equation}
Y_m(t,\theta) = \tilde a(t,X_m(t,\theta),\theta)
\end{equation}
for every $(t,\theta)\in[0,1]\times\Theta$ and
for every $m\in\NN\cup\{ \infty \}$, we 
deduce that \eqref{eq:conv_int_term} holds.
\end{proof}

We are now in position to prove that 
\eqref{eq:opt_ctrl_ens} admits a solution.

\begin{theorem}\label{thm:exist_minim}
Let $\F:\U\to\R_+$ be the functional defined in
\eqref{eq:def_fun_ens}. Then,
there exists $\hat u\in \U$ such that
\begin{equation*}
\F(\hat u) = \inf_\U \F.
\end{equation*}
\end{theorem}

\begin{proof}
We establish the thesis by means of the direct
method of calculus of variations (see, e.g.,
\cite[Theorem~1.15]{D93}).
Namely, we show that the functional 
$\F$ is coercive and lower semi-continuous
with respect to the weak topology of
$L^2$.
We first address the coercivity, i.e., we prove
that the sub-level sets of the functional
$\F$ are $L^2$-weakly pre-compact.
To see that, it is sufficient to
observe that for every
$M\geq0$ we have
\begin{equation} \label{eq:sublevel_inclusion}
\{ u\in\U: \F(u)\leq M \}\subset \{ u\in \U:
 ||u||_{L^2}^2 \leq 2M/\beta \},
\end{equation}
where we used the fact that the first term
at the right-hand side of \eqref{eq:def_fun_ens}
is non-negative.
To study the lower semi-continuity, let us consider
a sequence of admissible controls 
$(u_m)_{m\in\NN}\subset \U$
such that $u_m\weak_{L^2}u_\infty$ as $m\to\infty$.
Using the family of applications
$(Y_m)_{m\in\NN\cup\{\infty\}}$ defined as in
\eqref{eq:def_Y_m}, we observe that the integral
term at the right-hand side of 
\eqref{eq:def_fun_ens} can be rewritten as follows
\begin{equation*}
\int_\Theta\int_0^1 a(t,x_{u_m}^\theta(t),\theta)
\,d\nu(t)\,d\mu(\theta) =
\int_\Theta\int_0^1 Y_m(t,\theta)\,d\nu(t)\,d\mu
(\theta)
\end{equation*}
for every $m\in \NN\cup\{ \infty \}$.
Moreover, the uniform convergence
$Y_m\to_{C^0} Y_\infty$ as $m\to\infty$
provided by
Lemma~\ref{lem:conv_int_term} implies
in particular 
%that
%\begin{equation*}
%\lim_{m\to\infty}
%\int_\Theta\int_0^1 Y_m(t,\theta)\,d\nu(t)\,d\mu
%(\theta)
%= 
%\int_\Theta\int_0^1 Y_\infty(t,\theta)\,d\nu(t)\,d
%\mu (\theta), 
%\end{equation*}
%that in turn gives 
the convergence of the 
integral term at the right-hand side of
\eqref{eq:def_fun_ens}:
\begin{equation}\label{eq:conv_integrals}
\lim_{m\to\infty}
\int_\Theta\int_0^1 
a(t,x_{u_m}(t)^\theta,\theta)
\,d\nu(t)\,d\mu
(\theta)
=
\int_\Theta\int_0^1 
a(t,x_{u_\infty}(t)^\theta,\theta)
\,d\nu(t)\,d\mu
(\theta).
\end{equation}
Finally, 
combining \eqref{eq:norm_weak_semicont} with
\eqref{eq:conv_integrals}, we deduce that
\begin{equation*}
\F(u_\infty) \leq \liminf_{m\to\infty}\F(u_m).
\end{equation*}
This proves that  the functional $\F$
is lower semi-continuous, and therefore we 
obtain the thesis.
\end{proof}

\begin{remark}\label{rmk:role_beta}
The  constant $\beta>0$ in 
\eqref{eq:def_fun_ens} is aimed at
balancing the effect of 
the squared $L^2$-norm regularization and of
the integral term. This fact can be crucial in some
cases, relevant for applications. Indeed, let us 
assume that, for every $\e>0$, there exists
$u_\e\in \U$ such that 
\begin{equation*}
\int_\Theta 
\int_0^1a(t,x_{u_\e}^\theta(t),\theta)\,d\nu(t)
\,d\mu(\theta) \leq \frac\e2.
\end{equation*}
Then, let us set 
\begin{equation*}
\beta = \frac{\e}{||u_\e||_{L^2}^2},
\end{equation*}
and let $\hat u \in \U$ be a minimizer for 
the functional $\F:\U\to\R_+$ defined as in
\eqref{eq:def_fun_ens}. 
Therefore, we have that
\begin{equation*}
\int_\Theta 
\int_0^1a(t,x_{\hat u}^\theta(t),\theta)\,d\nu(t)
\,d\mu(\theta)
\leq \F(\hat u) \leq \F(u_\e) \leq \e.
\end{equation*}
In particular, this means that, when the constant
$\beta>0$ is chosen small enough, 
the integral cost achieved by the
minimizers of $\F$ can be made arbitrarily small. 
\end{remark}

\begin{remark} \label{rem:non_neg_cost}
The non-negativity assumption on the cost function $a:[0,1]\times\R^n\times\Theta\to\R_+$ is used to deduce the inclusion \eqref{eq:sublevel_inclusion}. This hypothesis can be relaxed by requiring, for example, that $a$ is bounded from below. More in general, our analysis is still valid for any continuous function $a:[0,1]\times\R^n\times\Theta\to\R$ such that the sublevels $\{ u\in\U: \F(u)\leq M \}\subset \U$ are bounded in $L^2$ for every 
$M\in \R$. For simplicity, we will assume throughout the paper that $a$ is non-negative.
\end{remark}

\end{section}

\begin{section}{Reduction to finite ensembles via
$\Gamma$-convergence} \label{sec:G_conv}

In this section we deal with the task of approximating infinite ensembles with growing-in-size finite ensembles, such that the minimizers of the corresponding ensemble optimal control problems are converging. 
In this framework, a natural attempt consists in approximating  the assigned probability measure $\mu$ on the space of parameters $\Theta$ with a probability measure $\bar \mu$ that charges a finite number of elements of $\Theta$.
Therefore, if $\mu$ and $\bar \mu$ are close in some appropriate sense, we may expect that the solutions of the minimization problem involving $\bar \mu$ provide approximations of the minimizers of the original ensemble optimal control problem \eqref{eq:opt_ctrl_ens}.
This argument can be made rigorous using the tools of $\Gamma$-convergence. We briefly recall below this notion. For a thorough introduction to this topic, we refer the reader to the textbook \cite{D93}.

\begin{defn} \label{defn:G_conv}
Let $(\X,d)$ be a metric space, and for every $N\geq 1$ let $\F^N:\X\to\R\cup \{+\infty \}$ be a functional defined over $X$. 
The sequence $(\F^N)_{N\geq 1}$ is said to $\Gamma$-converge to a functional $\F:\X\to\R\cup \{+\infty \}$ if the following conditions holds:
\begin{itemize}
\item \emph{liminf condition}: for every sequence
$(u_N)_{N\geq 1}\subset \X$ such that $u_N\to_\X u$
as $N\to\infty$ the following inequality holds
\begin{equation}\label{eq:liminf_cond}
\F(u) \leq \liminf_{N\to\infty} \F^N(u_N);
\end{equation}
\item \emph{limsup condition}: 
for every $x\in \X$ there exists a sequence 
$(u_N)_{N\geq 1}\subset \X$ such that 
$u_N\to_\X u$ as $N\to\infty$ and such that
the following inequality holds:
\begin{equation}\label{eq:limsup_cond}
\F(u) \geq \limsup_{N\to\infty} \F^N(u_N).
\end{equation}
\end{itemize}
If the conditions listed above are satisfied, then
we write $\F^N \to_\Gamma \F$ as $N\to\infty$.
\end{defn} 

The importance of the $\Gamma$-convergence is due to the fact that it relates the minimizers of the functionals $(\F^N)_{N\geq 1}$ to the minimizers of the limiting functional $\F$. Namely, under the hypothesis that the functionals of the sequence $(\F^N)_{N\geq 1}$ are equi-coercive, if $\hat u_N \in \mathrm{arg\, min}\, \F^N$ for every $N\geq 1$, then the sequence $(\hat u_N)_{N\geq1}$ is pre-compact in $\X$, and any of its limiting points is a minimizer for $\F$ (see \cite[Corollary~7.20]{D93}).
In other words, the problem of minimizing $\F$ can be approximated by the minimization of $\F^N$, when $N$ is sufficiently large. 

We now focus on the ensemble optimal control problem \eqref{eq:opt_ctrl_ens} studied in Section~\ref{sec:opt_ctrl_ens} and on the functional $\F:\U\to\R_+$ defined in \eqref{eq:def_fun_ens}. 
As done in the proof of Theorem~\ref{thm:exist_minim}, it is convenient to equip the space of admissible controls $\U:=L^2([0,1],\R^k)$ with the weak topology. However, Definition~\ref{defn:G_conv} requires the domain $\X$ where the limiting and the approximating functionals are defined to be a metric space. Unfortunately, the weak topology of $L^2$ is metrizable only when restricted to bounded sets (see, e.g., \cite[Remark~3.3 and Theorem~3.29]{B11}). 
In the next lemma we see how we should choose the restriction without losing any of the minimizers of $\F$.

\begin{lemma}\label{lem:restr}
Let $\F:\U\to\R_+$ be the functional defined in
\eqref{eq:def_fun_ens}. Therefore, there exists 
$\rho>0$ such that, if $\hat u \in \U$
satisfies $\F(\hat u)= \inf_\U \F$, then
\begin{equation} \label{eq:rad_restr}
||\hat u ||_{L^2} \leq \rho.
\end{equation} 
\end{lemma}

\begin{proof}
Let us consider the control $\bar u \equiv 0$. If 
$\hat u \in \U$ is a minimizer for 
the functional $\F$, then 
we have $\F(\hat u) \leq \F(\bar u)$.
Moreover, recalling that the function
$a:[0,1]\times\R^n\times \Theta\to\R_+$
that designs the integral cost in 
\eqref{eq:opt_ctrl_ens} is non-negative,
we deduce that
\[
\frac\be2 ||\hat u||_{L^2}^2 \leq \F(\hat u)
\leq \F(\bar u).
\] 
Thus, to prove \eqref{eq:rad_restr} it is sufficient
to set $\rho:= \sqrt{2\F(\bar u)/\be}$.
\end{proof}

The previous result implies that the following inclusion holds
\begin{equation*}
\mathrm{arg\, min}\, \F \subset \X,
\end{equation*}
where we set 
\begin{equation} \label{eq:def_restr_sp}
\X:=\{ u\in \U: ||u||_{L^2} \leq\rho\},
\end{equation}
and where $\rho>0$ is provided by Lemma~\ref{lem:restr}.
Since $\X$ is a closed ball of $L^2$, the weak topology induced on $\X$ is metrizable. Hence, we can restrict the functional $\F:\U\to\R_+$ to $\X$ to construct an approximation in the sense of $\Gamma$-convergence. With a slight abuse of notations, we shall continue to denote by $\F$ the functional restricted to $\X$. As anticipated at the beginning of the present section, the construction of the functionals $(\F^N)_{N\geq 1}$ relies on the introduction of a proper sequence of probability measures $(\mu_N)_{N\geq1}$ on $\Theta$ that approximate the probability measure $\mu$ prescribing the integral cost in \eqref{eq:opt_ctrl_ens}.  
We first recall the notion of weak convergence of probability measures. For further details, see, e.g., the textbook \cite[Definition~3.5.1]{C05}.

\begin{defn} \label{def:weak_conv_prob}
Let $(\mu_N)_{N\geq 1}$ be a sequence of Borel probability measures on the compact set $\Theta$. The sequence $(\mu_N)_{N\geq 1}$ is weakly convergent to the probability measure $\mu$ as $N\to\infty$ if the following identity holds
\begin{equation} \label{eq:weak_conv_meas}
\lim_{N\to\infty} \int_\Theta f(\theta)\, d\mu_N(\theta)
= 
\int_\Theta f(\theta)\, d\mu(\theta),
\end{equation}
for every function $f\in C^0(\Theta,\R)$. If the previous condition is satisfied, we write $\mu_N\weak^*\mu$ as $N\to\infty$.
\end{defn}

For every $N\geq1$ we consider a subset $\{ \theta_1,\ldots,\theta_N \}\subset \Theta$ and a probability measure that charges these elements:
\begin{equation} \label{eq:emp_meas}
\mu_N := \sum_{j=1}^N \alpha_j \delta_{\theta_j},
\qquad \mbox{ where } \qquad
\sum_{j=1}^N\alpha_j=1, \quad 
\alpha_j> 0 \quad \forall
j=1,\ldots,N.
\end{equation}
We assume that the sequence $(\mu_N)_{N\geq1}$ approximates the probability measure $\mu$ in the weak sense, i.e., we require that $\mu_N\weak^* \mu$ as $N\to\infty$.

\begin{remark} \label{rem:discr_prob_meas}
In the applications, there are several feasible strategies to achieve the convergence $\mu_N\weak^* \mu$ as $N\to\infty$, and the crucial aspect is whether the probability measure $\mu$ is explicitly known or not. If it is, the discrete approximating measures can be defined, for example, by following the construction proposed in \cite[Lemma~5.2]{BK19}. We observe that the problem of the \textit{optimal} approximation of a probability measure with a convex combination of a fixed number of Dirac deltas is an active research field. For further details, see, e.g.,  the recent paper \cite{MSS21}. On the other hand, in the practice, it may happen that there is no direct access to the probability measure $\mu$, but it is only possible to collect samplings of random variables distributed as $\mu$. In this case, the discrete approximating measures can be produced through a \textit{data-driven} approach. Namely, if $\{ \theta_1,\ldots,\theta_N \}$ are the empirically observed samplings, a natural choice is to set in \eqref{eq:emp_meas} $\alpha^j = \frac1N$ for every $j=1,\ldots,N$.
\end{remark}

We are now in position to introduce the family of functionals $(\F^N)_{N\geq 1}$. For every $N\geq 1$, let $\F^N:\X\to\R_+$ be defined as follows
\begin{equation} \label{eq:def_fun_N}
\F^N(u) := \int_\Theta\int_0^1
a(t,x^\theta_u(t),\theta) \,d\nu(t) d\mu_N(\theta)
+ \frac\be2 ||u||_{L^2}^2,
\end{equation}
where $x_u^\theta:[0,1]\to\R^n$ denotes the solution on \eqref{eq:ens_ctrl_sys} corresponding to the parameter $\theta\in \Theta$ and to the control $u\in \X$. We observe that $\F^N$ and $\F$ have essentially the same structure: the only difference is that the integral term of \eqref{eq:def_fun_ens} involves the measure $\mu$, while \eqref{eq:def_fun_N} features the measure $\mu_N$. 
Before proceeding to the main theorem of the section, we recall an auxiliary result.

\begin{lemma} \label{lem:conv_prod_meas}
Let $(\mu_N)_{N\geq1}$ be a sequence of probability
measures on $\Theta$ such that 
$\mu_N\weak^*\mu$ as $N\to\infty$, and let $\nu$
be a probability measure on $[0,1]$.
Then,
the sequence of the product probability measures
$(\nu\otimes \mu_N)_{N\geq1}$ on the
product space $[0,1]\times \Theta$ satisfies 
$\nu\otimes \mu_N\weak^* \nu\otimes \mu$
as $N\to\infty$.
\end{lemma}
\begin{proof}
The thesis follows directly
from Fubini Theorem and
Definition~\ref{def:weak_conv_prob}.
\end{proof}

We now show that the sequence of functionals
$(\F^N)_{N\geq1}$ introduced in \eqref{eq:def_fun_N}
is $\Gamma$-convergent to the functional that
defines the ensemble optimal control problem
\eqref{eq:opt_ctrl_ens}. 

\begin{theorem}\label{thm:G_conv}
Let $\X\subset \U$ be the set defined in
\eqref{eq:def_restr_sp}, equipped with the
weak topology of $L^2$. For every $N\geq1$, let
$\F^N:\X\to\R_+$ be the functional introduced
in \eqref{eq:def_fun_N}, and let $\F:\X\to\R_+$ be
the restriction to $\X$ of the application defined
in \eqref{eq:def_fun_ens}. Then, we have
$\F^N\to_\Gamma \F$ as $N\to\infty$.
\end{theorem}

\begin{proof}
We first establish the \emph{liminf condition}.
Let us consider a sequence of controls
$(u_N)_{N\geq1}\subset \X$ such that
$u_N\weak_{L^2} u_\infty$ as $N\to\infty$.
As done in Lemma~\ref{lem:conv_int_term},
for every $N\in \NN \cup\{ \infty \}$ let us 
define the functions $Y_N:[0,1]\times\Theta\to\R_+$
as follows:
\begin{equation} \label{eq:def_aux_Y_N}
Y_N(t,\theta) := a(t,X_N(t,\theta),\theta)
\end{equation}
for every $(t,\theta)\in[0,1]\times\Theta$, where,
for every $N\in \NN\cup \{\infty\}$,
$X_N:[0,1]\times\Theta\to\R^n$ is the mapping 
introduced in \eqref{eq:def_evol_ens} 
that describes the evolution of the ensemble
in correspondence of the admissible control
$u_N$. 
%Let us consider a subsequence $(u_{N_j})_{j\geq 1}$
%such that 
%\begin{equation*} %\label{eq:liminf_ineq_subs}
%C := \liminf_{N\to\infty} \F^N(u_N) =
%\liminf_{j\to\infty}\F_{N_j}(u_{N_j}).
%\end{equation*}
%Without loss of generality, we may assume that 
From \eqref{eq:def_aux_Y_N} and the definition
of the functionals $(\F^N)_{N\geq 1}$ in 
\eqref{eq:def_fun_N}, we obtain that 
\begin{equation} \label{eq:F_N_Y_N}
\F^N(u_N) = \int_\Theta \int_0^1 Y_N(t,\theta)
\,d\nu(t) d\mu_N(\theta) + \frac\be2||u_N||_{L^2}^2 
\end{equation}
for every $N\in \NN$. Moreover, we observe that
the uniform convergence 
$Y_N\to_{C^0}Y_\infty$ as $N\to\infty$ 
guaranteed by Lemma~\ref{lem:conv_int_term}
implies that
\begin{equation} \label{eq:int_cost_aux_1}
\lim_{N\to\infty} \int_\Theta\int_0^1
|Y_N(t,\theta) - Y_\infty(t,\theta)| \,d\nu(t)
d\mu_N(\theta) = 0.
\end{equation}
Therefore, using the triangular inequality
and Lemma~\ref{lem:conv_prod_meas}, 
from \eqref{eq:int_cost_aux_1} we deduce that
\begin{equation} \label{eq:int_cost_aux_2}
\lim_{N\to\infty} \int_\Theta\int_0^1
Y_N(t,\theta) 
\,d\nu(t) d\mu_N(\theta)
=
\int_\Theta\int_0^1
Y_\infty(t,\theta) \,d\nu(t)
d\mu(\theta).
\end{equation}
Combining  \eqref{eq:F_N_Y_N} with 
\eqref{eq:int_cost_aux_2} and 
\eqref{eq:norm_weak_semicont}, we 
have that
\begin{equation*}
\F(u_\infty) \leq \liminf_{N\to\infty} \F^N(u_N),
\end{equation*}
which concludes the first part of the proof.

We now establish the \emph{limsup condition}.
For every $u\in \X$, let us consider the constant
sequence $u_N = u$ for every $N\geq 1$.
In virtue of Lemma~\ref{lem:conv_prod_meas}, we 
have that
\begin{equation} \label{eq:limsup_aux}
\lim_{N\to\infty} \int_\Theta\int_0^1
a(t,X_u(t,\theta),\theta)\, d\nu(t)d\mu_N(\theta)
= 
\int_\Theta\int_0^1
a(t,X_u(t,\theta),\theta)\, d\nu(t)d\mu(\theta)
\end{equation}
for every $u\in\X$, where $X_u:[0,1]\times\Theta
\to\R^n$ is defined as in \eqref{eq:def_evol_ens}.
This fact gives
\begin{equation*}
\F(u) = \lim_{N\to\infty}\F^N(u)
\end{equation*}
for every $u\in\X$, and this shows that the 
\emph{limsup condition} holds.
\end{proof}

\begin{remark} \label{rmk:exist_min_F_N}
We observe that Theorem~\ref{thm:exist_minim} holds also for $\F^N:\X\to\R_+$ for every $N\in\NN$. Indeed, the domain $\X$ is itself sequentially weakly compact, and the convergence \eqref{eq:conv_integrals} occurs also with the probability measure $\mu_N$ in place of $\mu$.
Therefore, as the functional is $\F^N$ coercive and sequentially lower semi-continuous with respect to the weak topology of $L^2$, it admits a minimizer.
\end{remark}

The next result is a direct consequence of the $\Gamma$-convergence result established in Theorem~\ref{thm:G_conv}. Indeed, as anticipated before, the fact that the minimizers of the functionals $(\F^N)_{N\in\NN}$ provide approximations of the minimizers of the limiting functional $\F$ is a well-established fact, as well as the convergence $\inf_\X \F^N \to \inf_\X \F$ as $N\to\infty$ (see \cite[Corollary~7.20]{D93}).
We stress the fact that, usually, the approximation of the minimizers occurs in the topology that underlies the $\Gamma$-convergence result. However, we can actually prove that, in this case, the approximation is provided with respect to the \emph{strong} topology of $L^2$, and not just in the weak sense. Similar phenomena have been recently described in \cite[Theorem~7.4]{S1} and in \cite[Remark~6]{S2}.

\begin{corollary} \label{cor:conv_min}
Let $\X\subset \U$ be the set defined in
\eqref{eq:def_restr_sp}. For every $N\geq1$, let
$\F^N:\X\to\R_+$ be the functional introduced
in \eqref{eq:def_fun_N} and let $\hat u_N \in \X$
be any of its minimizers. 
Finally, let $\F:\X\to\R_+$ be
the restriction to $\X$ of the application defined
in \eqref{eq:def_fun_ens}.
Then, we have 
\begin{equation} \label{eq:conv_inf}
\inf_\X \F = \lim_{N\to\infty} \inf_\X\F^N.
\end{equation} 
Moreover, the sequence
$(\hat u_N)_{N\in\NN}$ is pre-compact with
respect to the strong topology of $L^2$, and 
any limiting point of this sequence is a minimizer
of $\F$.
\end{corollary}

\begin{proof}
Owing to Theorem~\ref{thm:G_conv}, we have that
$\F^N\to_\Gamma \F$ as $N\to\infty$ with respect
to the weak topology of $L^2$. Therefore, 
from \cite[Corollary~7.20]{D93} it follows
that \eqref{eq:conv_inf} holds and that
the sequence of minimizers $(\hat u_N)_{N\in\NN}$
is pre-compact with respect to the weak topology
of $L^2$, and its limiting points are minimizers
of $\F$.
To conclude we have to prove that it is pre-compact
with respect to the strong topology, too.
Let us consider a subsequence 
$(\hat u_{N_j})_{j\in\NN}$
such that $\hat u_{N_j}\weak_{L^2}\hat u_\infty$
as $j\to\infty$. Using the fact that 
$\hat u_\infty$ is a minimizer for $\F$,
as well as $\hat u_{N_j}$ is for $\F^{N_j}$
for every $j\in\NN$, from \eqref{eq:conv_inf}
it follows that
\begin{equation} \label{eq:conv_val_min}
\F(\hat u_\infty) = \lim_{j\to\infty} \F^{N_j}
(\hat u_{N_j}).
\end{equation}
Moreover, with the same argument used 
in the proof of Theorem~\ref{thm:G_conv}
to deduce the identity \eqref{eq:int_cost_aux_2},
we obtain that
\begin{equation} \label{eq:int_cost_aux_3}
\int_\Theta\int_0^1 a(t,x_{\hat u_\infty}^\theta(t),
\theta)\, d\nu(t)d\mu(\theta) = 
\lim_{j\to\infty}
\int_\Theta\int_0^1 a(t,x_{\hat u_{N_j}}^\theta(t),
\theta)\, d\nu(t)d\mu_{N_j}(\theta).
\end{equation}
Combining \eqref{eq:conv_val_min} and 
\eqref{eq:int_cost_aux_3}, and recalling the 
definitions \eqref{eq:def_fun_N} and 
\eqref{eq:def_fun_ens}
of the functionals 
$\F^N:\X\to\R_+$ and $\F:\X\to\R_+$,
we have that
\begin{equation*}
\frac\be2 ||\hat u_\infty||_{L^2}^2 
= \lim_{j\to\infty} 
\frac\be2 ||\hat u_{N_j}||_{L^2}^2,
\end{equation*} 
which implies that 
$\hat u_{N_j}\to_{L^2}\hat u_\infty$
as $j\to\infty$. Since the argument holds for
every $L^2$-weakly convergent subsequence 
of the sequence of minimizers 
$(\hat u_N)_{N\in\NN}$, this concludes the proof.
\end{proof}

\begin{remark} \label{rem:G_conv_interpr}
There are two possible interpretations for Theorem~\ref{thm:G_conv} and Corollary~\ref{cor:conv_min}, depending if the probability measure $\mu$ that defines the limiting functional $\F$ is explicitly known or not. 
If it is, then the $\Gamma$-convergence result can be read as a theoretical guarantee to substitute an infinite-ensemble optimal control problem with a finite-ensemble one, as illustrated in the Introduction and at the beginning of this section.
On the other hand, in real-world problems, the underlying measure $\mu$ may be unknown, but we can collect observations $\{ \theta_1,\ldots,\theta_N \}$ of random variables distributed as $\mu$, and we consider the empirical probability measure $\mu_N=\frac1N \sum_{j=1}^N\delta_{\theta_i}$. In this framework, Theorem~\ref{thm:G_conv} and Corollary~\ref{cor:conv_min} can be interpreted as \textit{stability results} for the number of observations $N$. Indeed, from the fact that $\mu_N\weak^* \mu$ as $N\to\infty$, the $\Gamma$-convergence of the sequence $(\F^N)_{N\in\NN}$ implies that, when the number of collected observations is large enough, we should not expect dramatic changes in the solutions of the optimal control problems if we further increase the samplings. 
\end{remark}

\end{section}

\begin{section}{Gradient field
and Maximum Principle for the approximating
problems} \label{sec:grad_flow}

In the present section we address the question of actually finding the minimizers of the approximating functionals $(\F^N)_{N\in\NN}$ introduced in Section~\ref{sec:G_conv}. Namely, starting from the result stated in Theorem~\ref{thm:grad_field_aff} for a {\it single} affine-control system with end-point cost, we obtain the expression of the gradient fields that the functionals $(\F^N)_{N\in\NN}$ induce on their domain. Moreover, we state the Pontryagin Maximum Principle for the optimal control problems corresponding to the minimization of the functionals $(\F^N)_{N\in \NN}$. Both the gradient fields and the Maximum Principle will be used for the construction of the numerical algorithms presented in Section~\ref{sec:num_schemes}.

From now on, we specialize on the following particular form of the cost associated with the ensemble optimal control problem
\eqref{eq:opt_ctrl_ens}:
\begin{equation} \label{eq:fun_ens_endpoint}
\F(u) = \int_{\Theta} a(x_u^\theta(1),\theta)
\, d\mu(\theta) + \frac\beta2||u||_{L^2}^2
\end{equation} 
for every $u\in\U$,
where $a:\R^n\times\Theta\to\R_+$ is a $C^1$-regular function, and $\beta>0$ is a positive parameter that tunes the $L^2$-regularization.
We observe that \eqref{eq:fun_ens_endpoint} is a particular instance of \eqref{eq:def_fun_ens}. Indeed, it corresponds to the case $\nu= \delta_{t=1}$, where $\nu$ is the probability measure on the time interval $[0,1]$ that appears in the first term at the right-hand side of \eqref{eq:opt_ctrl_ens}.
In other words, we assume that the integral cost in \eqref{eq:opt_ctrl_ens} depends only on the final state of the trajectories of the ensemble. 
For every $N\in \NN$, let the probability measure $\mu_N$ have the same expression as in \eqref{eq:emp_meas}, i.e., it is a finite convex combination of Dirac deltas centered at $\{ \theta_1,\ldots,\theta_N \} \subset \Theta$. 
Therefore, for every $N\in\NN$, the functional $\F^N:\U\to\R_+$ that we consider in place of \eqref{eq:fun_ens_endpoint} has the the form
\begin{equation} \label{eq:fun_ens_N_endpoint}
\F^N(u)=
\int_\Theta a(x_u^\theta(1),\theta)\,d\mu_N(\theta)+
\frac\beta2 ||u||_{L^2}^2
= \sum_{j=1}^N \alpha_j
 a(x_u^{\theta_j}(1),\theta_j) +
 \frac{\beta}{2}||u||_{L^2}^2
\end{equation} 
for every $u\in\U$.

\begin{remark} \label{rmk:dom_def_1}
In Section~\ref{sec:G_conv} for technical reasons we defined the functionals $(\F^N)_{N\in\NN}$ on the domain $\X\subset\U$ introduced in \eqref{eq:def_restr_sp}. However, the functionals $(\F^N)_{N\in\NN}$ and the corresponding gradient fields can be defined over the whole space of admissible controls $\U$. 
\end{remark}

At this point, it is convenient to approach the minimization of the functional $\F^N$ in the framework of finite-dimensional optimal control problems in finite-dimensional Euclidean spaces.
For this purpose, we introduce some notations. For every $N\in\NN$, let $\{ \theta_1,\ldots,\theta_N \} \subset \Theta$ be the set of parameters charged by the discrete probability measure $\mu_N$.
Then, we study the finite sub-ensemble of \eqref{eq:ens_ctrl_sys} corresponding to the parameters $\{ \theta_1,\ldots,\theta_N \}$. Namely, we consider the following affine-control system on $\R^{nN}$:
\begin{equation} \label{eq:sys_ctrl_ens_discr}
\begin{cases}
\dot \xx_u(t) = \FF_0^{N}(\xx_u)
+ \FF^{N}(\xx_u)u(t), &
\mbox{for a.e. }t\in[0,1],\\
\xx_u(0) = \xx_0,
\end{cases}
\end{equation}
where $\xx = (x^1,\ldots,x^N)^T\in \R^{nN}$, and $\FF_0^N:\R^{nN}\to\R^{nN}$ and $\FF^N:\R^{nN}\to\R^{nN\times k}$ are applications defined as follows:
\begin{equation} \label{eq:def_bf_field_0}
\FF_0^N(\xx) := \left( \begin{matrix}
F_0^{\theta_1}(x^1)\\
\vdots\\
F_0^{\theta_N}(x^N)
\end{matrix}
\right)
\end{equation}
and
\begin{equation}\label{eq:def_bf_fields}
\FF^N(\xx) := \left( \begin{matrix}
F^{\theta_1}(x^1)\\
\vdots\\
F^{\theta_N}(x^N)
\end{matrix}
\right)=
\left( \begin{matrix}
F_1^{\theta_1}(x^1)&\ldots &F_k^{\theta_1}(x^1)\\
\vdots&&\vdots\\
F_1^{\theta_N}(x^N)&\ldots& F_k^{\theta_N}(x^N)
\end{matrix}
\right)
\end{equation}
for every $\xx\in\R^{nN}$. Finally, the initial value is set as $\xx_0 := (x_0(\theta_1),\ldots,x_0(\theta_N))$, where $x_0:\Theta\to\R^n$ is the mapping defined \eqref{eq:def_x0} that prescribes the initial data of the Cauchy problems of the ensemble \eqref{eq:ens_ctrl_sys}.
Moreover, we can introduce the function $\bfa^N:\R^{nN}\to\R_+$ defined as 
\begin{equation} \label{eq:def_bf_a}
\bfa^N(\xx)= \bfa^N((x^1,\ldots,x^N)):= 
\sum_{j=1}^N \alpha_j a(x^j,\theta_j),
\end{equation}
where $a:\R^n\times\Theta\to\R_+$ is the function that designs the integral cost in \eqref{eq:fun_ens_endpoint}, and for every $j=1,\ldots,N$ the coefficient $\alpha_j$ is the weight corresponding to $\delta_{\theta_j}$ in the convex combination \eqref{eq:emp_meas}.
In this framework, the functional $\F^N:\U\to\R_+$ can be rewritten as follows:
\begin{equation} \label{eq:fun_ens_N_rewr}
\F^N(u) = \bfa(\xx_u(1)) + 
\frac\beta2 ||u||^2_{L^2}
\end{equation} 
for every $u\in\U$, where $\xx_u^N:[0,1]\to\R^{nN}$
is the solution of \eqref{eq:sys_ctrl_ens_discr}
corresponding to the admissible control $u$.
In the next result we derive the expression of
the vector field $\G^N:\U\to\U$ that represents
the differential of the functional $\F^N$, i.e., that satisfies
\begin{equation} \label{eq:rep_diff_fun_ens}
\langle\G^N[u],v
\rangle_\U = d_u\F^N(v)
\end{equation}
for every $u,v\in \U$.

\begin{theorem} \label{thm:grad_field_N}
Let us assume that for every $\theta\in\Theta$
the functions $x\mapsto F_0(x,\theta)$
and $x\mapsto F(x,\theta)$ are $C^1$-regular,
as well as the function $x\mapsto a(x,\theta)$
that defines the end-point cost in 
\eqref{eq:fun_ens_endpoint}.
Let $\{ \theta_1,\ldots,\theta_N \}
\subset \Theta$ be the subset of parameters
charged by the measure $\mu^N$ that designs
the integral cost in \eqref{eq:fun_ens_N_endpoint}.
Let $\F^N:\U\to\R_+$ be the functional defined in
\eqref{eq:fun_ens_N_endpoint}. Then, $\F^N$
is Gateaux differentiable at every $u\in\U$,
and we define 
$\G^N:\U\to\U$ as the gradient vector field
on $\U$ that satisfies \eqref{eq:rep_diff_fun_ens}.
Then, for every $u\in\U$ we have 
\begin{equation} \label{eq:grad_field_N}
\G^N[u](t) = \sum_{j=1}^N \alpha_j F^{\theta_j}
(x_u^{\theta_j}(t))^T \cdot \lambda_u^{j\,T}(t)
+\beta u(t)
\end{equation}
for a.e. $t\in[0,1]$, where for every 
$j=1,\ldots, N$ the curve 
$x_u^{\theta_j}:[0,1]\to\R^n$ is the solution of
\eqref{eq:ens_ctrl_sys} corresponding to the
parameter $\theta_j$ and to the admissible control
$u$, and $\lambda_u^j:[0,1]\to(\R^n)^*$ is 
the absolutely continuous curve of covectors 
that solves
\begin{equation}\label{eq:adj_eq_N}
\begin{cases}
\dot \lambda_u^j(t) = -\lambda_u^j(t)
\left(\frac{\partial F_0^{\theta_j}(x_u^{\theta_j}(t))}{\partial x} + 
\sum_{i=1}^k u_i(t)
\frac{\partial F_i^{\theta_j}(x_u^{\theta_j}(t))}{\partial x}
\right) &\mbox{a.e. in } [0,1],\\
\lambda_u^j(1) = \nabla a(x_u^{\theta_j}(1),\theta_j).
\end{cases}
\end{equation}
\end{theorem}

\begin{remark}\label{rmk:row_vect}
We use the convention that the elements of $(\R^n)^*$
are row-vectors. Therefore, for every
$j=1,\ldots,N$ and $t\in[0,1]$, $\lambda_u^j(t)$
should be read as a row-vector. 
This should be considered to give sense to 
\eqref{eq:grad_field_N} and \eqref{eq:adj_eq_N}.
The same observation holds for 
Theorem~\ref{thm:PMP_N}.
\end{remark}

\begin{proof}[Proof of Theorem~\ref{thm:grad_field_N}]
As done in \eqref{eq:sys_ctrl_ens_discr},
we can equivalently rewrite the sub-ensemble 
of control systems corresponding to 
the parameters $\{ \theta_1,\ldots,\theta_N \}
\subset \Theta$ as a {\it single} affine-control
system in $\R^{nN}$.
Moreover, the regularity hypotheses guarantee
that the functions
$\FF_0^N:\R^{nN}\to\R^{nN}$ and 
$\FF^N:\R^{nN}\to\R^{nN\times k}$ defined in 
\eqref{eq:def_bf_fields} are $C^1$-regular,
as well as the function $\bfa:\R^{nN}\to\R_+$
introduced in \eqref{eq:def_bf_a}.
Therefore, owing to Theorem~\ref{thm:grad_field_aff},
we obtain the expression for the
gradient field induced by the functional
$\F^N$ written in \eqref{eq:fun_ens_N_rewr}. 
Indeed, we deduce that
\begin{equation} \label{eq:grad_field_N_aux}
\G^N[u] = \FF^N(\xx_u(t))^T
\mathbf{\Lambda}_u(t) + \beta u
\end{equation}
for every $u\in\U$, where 
$\xx_u:[0,1]\to \R^{nN}$ is the solution of 
\eqref{eq:sys_ctrl_ens_discr} corresponding to
the control $u$, and $\mathbf{\Lambda}_u:[0,1]\to
(\R^{nN})^*$ is the curve of covectors that solves
\begin{equation} \label{eq:adj_eq_N_aux}
\begin{cases}
\dot{\mathbf{\Lambda}}_u(t) = -{\mathbf{\Lambda}}_u(t)
\left(
\frac{\partial \FF_0^N(\xx_u(t))}{\partial \xx} + \sum_{i=1}^k u_i(t)
\frac{\partial \FF_i^N(\xx_u(t))}{\partial \xx}
\right) &\mbox{for a.e. }t\in[0,1],\\
{\mathbf{\Lambda}}_u(1) = \nabla_{\xx} 
\bfa(\xx_u(1)),
\end{cases}
\end{equation}
where $\FF_1^N,\ldots,\FF_k^N:\R^{nN}\to\R^{nN}$
denote the vector fields obtained by taking the 
columns of the matrix-valued application
$\FF^N:\R^{nN}\to\R^{nN\times k}$.
Moreover, if we consider the curves of covectors
$\lambda_u^1,\ldots,\lambda_u^N:[0,1]\to(\R^n)^*$
that solve \eqref{eq:adj_eq_N} for  
$j=1,\ldots,N$, it turns out that 
the solution of \eqref{eq:adj_eq_N_aux}
can be written as 
$\mathbf{\Lambda}_u(t)=(\alpha_1 \lambda_u^1(t),\ldots,
\alpha_N\lambda_u^N(t))$ for every 
$t\in[0,1]$, where $\alpha_1,\ldots,\alpha_N$
are the coefficients of convex combination
involved in the definition of $\mu_N$ 
\eqref{eq:emp_meas}.
Finally, owing to this decoupling of
$\mathbf{\Lambda}_u$, the identity 
\eqref{eq:adj_eq_N} can be deduced from 
\eqref{eq:grad_field_N_aux}
using the expression of $\FF_0^N,\ldots,\FF_k^N$. 
\end{proof}

In the previous result we obtained the
Riesz's representation of the differential
of the functional $\F^N:\U\to\R_+$. 
We now establish the necessary condition for 
an admissible control $\hat u_N\in\U$ to be
a minimizer of $\F^N$. This essentially descends 
as a standard application of Pontryagin Maximum 
Principle. For a complete survey on the topic, the
reader is referred to the textbook \cite{AS}.

\begin{theorem}\label{thm:PMP_N}
Under the same assumptions and notations of 
Theorem~\ref{thm:grad_field_N}, let $\hat u_N
=(\hat u_{N,1},\ldots, \hat u_{N,k})\in \U$
be a local minimizer of the functional $\F^N:\U\to\R_+$
defined as in \eqref{eq:fun_ens_N_endpoint}.
For every $j=1,\ldots,N$, let 
$x_{\hat u_N}^{\theta_j}:[0,1]\to\R^n$  be the 
solution of \eqref{eq:ens_ctrl_sys} corresponding
to the parameter $\theta_j\in\Theta$ and to the
optimal control $\hat u_N$. Then,
for every $j=1,\ldots,N$ there 
exists a curve of covectors $\lambda_{\hat u_N}^j
:[0,1]\to(\R^n)^*$ such that
\begin{equation}\label{eq:adj_eq_PMP_N}
\begin{cases}
\dot \lambda_{\hat u_N}^j(t) = -\lambda_{\hat u_N}^j(t)
\left(\frac{\partial F_0^{\theta_j}(x_{\hat u_N}^{\theta_j}(t))}{\partial x} + 
\sum_{i=1}^k \hat u_{N,i} (t)
\frac{\partial F_i^{\theta_j}(x_{\hat u_N}^{\theta_j}(t))}{\partial x}
\right) &\mbox{a.e. in } [0,1],\\
\lambda_{\hat u_N}^j(1) = \nabla a(x_{\hat u_N}^{\theta_j}(1),\theta_j),
\end{cases}
\end{equation}
and such that 
\begin{equation}\label{eq:max_cond}
\hat u_N(t) \in 
\arg \max_{\!\!\!\!\!\!
\!\!\!\!\!\! v\in \R^k}
\left\{
\sum_{j=1}^N \alpha_j
\left( 
-\lambda_{\hat u_N}^j(t) \cdot
F^{\theta_j}(x_{\hat u_N}^{\theta_j}(t))\cdot v \right)
- \frac{\beta}{2}|v|^2_2
\right\}
\end{equation}
for a.e. $t\in[0,1]$.
\end{theorem}

\begin{proof}
As done in the proof of 
Theorem~\ref{thm:grad_field_N}, we observe that
we can equivalently
consider the {\it single} affine-control
system \eqref{eq:sys_ctrl_ens_discr} 
in place of the sub-ensemble of affine-control
systems corresponding to the parameters
$\{ \theta_1,\ldots,\theta_N \}\subset \Theta$.
Moreover, if we rewrite the cost functional
$\F^N:\U\to\R_+$ as in \eqref{eq:fun_ens_N_rewr},
we reduce to a standard optimal control
problem in $\R^{nN}$. Let $\hat u_N\in\U$  be an
optimal control for this problem, and
let $\xx_{\hat u_N}:[0,1]\to\R^{nN}$ be the 
solution of \eqref{eq:sys_ctrl_ens_discr} 
corresponding to $\hat u_N$. Then, from the 
Pontryagin Maximum Principle (see, e.g.,
\cite[Chapter~12]{AS}), there
exists $\epsilon\in\{0,-1 \}$ and 
$\mathbf{\Lambda}_{\hat u_N}:[0,1]\to(\R^{nN})^*$
such that $(\epsilon,\mathbf{\Lambda}_{\hat u_N}(t))
\neq 0$ for every $t\in[0,1]$ and such that
\begin{equation} \label{eq:adj_eq_PMP_N_aux}
\begin{cases}
\dot{\mathbf{\Lambda}}_{\hat u_N}(t) = -{\mathbf{\Lambda}}_{\hat u_N}(t)
\left(
\frac{\partial \FF_0^N(\xx_{\hat u_N}(t))}{\partial \xx} + \sum_{i=1}^k {\hat u_{N,i}}(t)
\frac{\partial \FF_i^N(\xx_{\hat u_N}(t))}{\partial \xx}
\right) &\mbox{a.e. in } [0,1],\\
{\mathbf{\Lambda}}_{\hat u_N}(1) = 
\epsilon \nabla_{\xx} 
\bfa(\xx_{\hat u_N}(1)).
\end{cases}
\end{equation}
Moreover, for a.e. $t\in[0,1]$ 
the following condition holds
\begin{equation}\label{eq:max_cond_aux}
\hat u_N(t) \in 
\arg \max_{\!\!\!\!\!\!
\!\!\!\!\!\! v\in \R^k}
\left\{
\mathbf{\Lambda}_{\hat u_N}(t)
\big(
\FF_0^N(\xx_{\hat u_N}(t))
+ \FF^N(\xx_{\hat u_N}(t))v
\big)
+ \epsilon\frac{\beta}{2}|v|^2_2
\right\}.
\end{equation}
Since the differential equation 
\eqref{eq:adj_eq_PMP_N_aux} is linear, 
if $\epsilon=0$ we have 
$\mathbf{\Lambda}_{\hat u_N}(t)\equiv 0$,
and this violates the 
condition $(\epsilon,\mathbf{\Lambda}_{\hat u_N}(t))
\neq 0$ for every $t\in[0,1]$. Therefore
we deduce that $\epsilon=-1$.
This shows that
the optimal control problem in consideration has
no abnormal extremals.
Moreover, if we consider the curves of covectors
$\lambda_{\hat u_N}^1,
\ldots,\lambda_{\hat u_N}^N:[0,1]\to(\R^n)^*$
that solve \eqref{eq:adj_eq_PMP_N} for  
$j=1,\ldots,N$, it turns out that 
the solution of \eqref{eq:adj_eq_PMP_N_aux}
corresponding to $\alpha =-1$
can be written as 
$\mathbf{\Lambda}_u(t)=(-\alpha_1\lambda_{\hat u_N}^1(t),
\ldots,-\alpha_N\lambda_{\hat u_N}^N(t))$ for every 
$t\in[0,1]$, where $\alpha_1,\ldots,\alpha_N$
are the coefficients of convex combination
involved in the definition of $\mu_N$ 
\eqref{eq:emp_meas}.
Finally, owing to this decoupling of
$\mathbf{\Lambda}_u$, the condition
\eqref{eq:max_cond} can be deduced from 
\eqref{eq:max_cond_aux}
using the expression of $\FF_0^N,\ldots,\FF_k^N$,
and observing that the term
$\mathbf{\Lambda}_{\hat u_N}(t)
\FF_0^N(\xx_{\hat u_N}(t))$ in 
\eqref{eq:max_cond_aux} does not affect the
minimizer.
\end{proof}

\begin{remark} \label{rmk:PMP_formulation}
We can equivalently reformulate the Maximum condition
\eqref{eq:max_cond} of Theorem~\ref{thm:PMP_N}
as follows:
\begin{equation}\label{eq:max_cond_Int}
\hat u_N(t) \in 
\arg \max_{\!\!\!\!\!\!
\!\!\!\!\!\! v\in \R^k}
\left\{
\int_\Theta -\Lambda_{\hat u_N}(t,\theta)\cdot
F(X_{\hat u_N}(t,\theta),\theta) \cdot
v \, d\mu_N(\theta)
- \frac\beta2 |v|_2^2 
\right\},
\end{equation}
where 
$\Lambda_{\hat u_N}:[0,1]\times\Theta\to (\R^n)^*$
and 
$X_{\hat u_N}:[0,1]\times\Theta\to \R^n$
are the applications defined, respectively, in
\eqref{eq:def_Lambda} and \eqref{eq:def_evol_ens},
and corresponding to the control $\hat u_N$.

\end{remark}

We recall that the Pontryagin Maximum 
Principle provides {\it necessary} condition
for minimality. An admissible control 
$\bar u \in \U$ is a (normal) {\it Pontryagin
extremal} for the optimal control problem 
related to the minimization of $\F^N:\U\to\R_+$
if there exist $\lambda_{\bar u}^1,\ldots,
\lambda_{\bar u}^N:[0,1]\to(\R^n)^*$
satisfying \eqref{eq:adj_eq_PMP_N} and such that
the relation \eqref{eq:max_cond} holds.

\begin{remark}\label{rmk:grad_PMP}
Let $\bar u\in \U$ be a critical point for
the functional $\F^N:\U\to\R_+$, i.e., 
$\G^N[\bar u]=0$. Therefore, from
\eqref{eq:grad_field_N} it turns out that
\begin{equation*} %\label{eq:expr_crit_F_N}
\bar u(t) = -\frac1\beta \sum_{j=1}^N
F^{\theta_j}(x_{\bar u}^{\theta_j}(t))^T\cdot 
\lambda_{\bar u}^j(t)^T
\end{equation*}
for a.e. $t\in[0,1]$, where for every $j=1,\ldots,N$
the curve
$x_{\bar u}^{\theta_j}:[0,1]\to\R^n$ is 
the trajectory of \eqref{eq:ens_ctrl_sys}
corresponding to the parameter $\theta_j$ and
to the control $\bar u$, and 
$\lambda_{\bar u}^j:[0,1]\to (\R^n)^*$
is the solution of \eqref{eq:adj_eq_N}.
We observe that, for every $j=1,\ldots,N$,
$\lambda_{\bar u}^j:[0,1]\to (\R^n)^*$ solves
as well \eqref{eq:adj_eq_PMP_N}, and that
$\bar u(t)$ satisfies 
\begin{equation*}
\bar u^N(t) \in 
\arg \max_{\!\!\!\!\!\!
\!\!\!\!\!\! v\in \R^k}
\left\{
\sum_{j=1}^N \alpha_j
\left(  -\lambda_{\bar u^N}^j(t)\cdot 
F^{\theta_j}(x_{\bar u^N}^{\theta_j}(t))\cdot v \right)
- \frac{\beta}{2}|v|^2_2
\right\}
\end{equation*}
for a.e. $t\in[0,1]$. This shows that any critical
point of $\F^N:\U\to\R_+$ is a (normal) Pontryagin
extremal for the corresponding optimal control
problem. Conversely, an analogue argument shows
that any Pontryagin extremal is a critical point
for the functional $\F^N$. 
\end{remark}

\end{section}

\begin{section}{Maximum Principle for ensemble optimal control problems} \label{sec:PMP_ens}
In the present section we use a $\Gamma$-convergence argument to recover necessary optimality conditions for (local) minimizers of the functional $\F$ defined in \eqref{eq:fun_ens_endpoint}.
The result that we prove here is in the same flavor as the Maximum Principle derived in \cite{BK19}, even though the tools employed are rather different.

Let $\bar u\in \U$ be a local minimizer for the functional $\F$. Then, for every $\e>0$, we define the following perturbed functional $\F^\e:\U\to\R_+$:
\begin{equation}\label{eq:def_F_eps}
\F^\e(u) := \int_{\Theta} a(x_u^\theta(1),\theta)
\, d\mu(\theta) + \frac\beta2||u||_{L^2}^2
+ \frac\e2 ||u-\bar u||_{L^2}^2.
\end{equation}
We immediately observe that the following property holds.

\begin{lemma} \label{lem:uniq_loc_min}
Let us consider the functional $\F:\U\to \R_+$
introduced in \eqref{eq:fun_ens_endpoint} and let
$\bar u\in \U$ be one of its local minimizers.
Let $\F^\e:\U\to\R_+$ be defined as in \eqref{eq:def_F_eps}. Then, there exists $\rho_{\bar u}>0$
such that  
\begin{equation*}
\F^\e(\bar u) < \F^\e( u)\quad   \forall u\neq \bar u \mbox{ with } ||u- \bar u||_{L^2}\leq \rho_{\bar u}
\end{equation*}
for every $\e>0$.
\end{lemma}
\begin{proof}
Since $\bar u$ is a local minimizer for $\F$, there
exists $\rho_{\bar u}>0$ such that $\F(\bar u) \leq \F(u)$ for
every $u$ satisfying $||u-\bar u||_{L^2}\leq \rho_{\bar u}$.
From the definition of $\F^\e$ in \eqref{eq:def_F_eps}
and observing that $\F^\e(\bar u) = \F(\bar u)$, 
we deduce the thesis.
\end{proof}

For every local minimizer $\bar u\in \U$ of the 
functional $\F$, we set
\begin{equation} \label{eq:def_dom_restr_u_bar}
\X_{\bar u}:= \{ x\in \U\mid 
||u-\bar u||_{L^2}\leq \rho_{\bar u} \}.
\end{equation}
Given a sequence of discrete probability measures 
$(\mu_N)_{N\geq 1}$ as in \eqref{eq:emp_meas} such that
$\mu_N\weak^* \mu$ as $N\to \infty$, for every $\e>0$
and for every $N\geq1$ we introduce the functional
$\F^{N,\e}:\X_{\bar u}\to\R_+$ as follows:
\begin{equation}\label{eq:def_F_N_eps}
\F^{N,\e}(u):=
\int_{\Theta} a(x_u^\theta(1),\theta)
\, d\mu_N(\theta) + \frac\beta2||u||_{L^2}^2
+ \frac\e2 ||u-\bar u||_{L^2}^2.
\end{equation}
Similar to Section~\ref{sec:G_conv},
we can establish a $\Gamma$-convergence result.

\begin{proposition}\label{prop:G_conv_F_eps}
Let $\bar u \in \U$ be a local minimizer of the 
functional $\F:\U\to\R_+$ introduced in 
\eqref{eq:fun_ens_endpoint}, and let
 $\X_{\bar u}\subset \U$ be the set defined in
\eqref{eq:def_restr_sp}, equipped with the
weak topology of $L^2$. For every $N\geq1$ and for every $\e>0$, let
$\F^{N,\e}:\X_{\bar u}\to\R_+$
 be the functional presented
in \eqref{eq:def_F_N_eps}, and let 
$\F^\e:\X_{\bar u}\to\R_+$ be
the restriction to $\X_{\bar u}$
 of the application defined
in \eqref{eq:def_F_eps}. Then, we have that
$\F^{N,\e}\to_\Gamma \F^\e$ as $N\to\infty$.
Moreover, if for every $N\geq1$ we consider 
$\hat u_{N,\e} \in \mathrm{argmin} \F^{N,\e}$,
we obtain that
\begin{equation} \label{eq:strong_conv_u_N_eps}
\lim_{N\to\infty} ||\hat u_{N,\e} -\bar u||_{L^2} =0.
\end{equation}
\end{proposition}
\begin{proof}
The fact that 
$\F^{N,\e}\to_\Gamma \F^\e$ as $N\to\infty$ follows
from a \textit{verbatim} repetition of the arguments
of the proof of Theorem~\ref{thm:G_conv}.
In addition, \cite[Corollary~7.20]{D93} guarantees that
\begin{equation} \label{eq:conv_min_F_eps}
\lim_{N\to\infty} \F^{N,\e}(\hat u_{N,\e}) =
\inf_{\X_{\bar u}}\F^\e =\F^\e(\bar u),
\end{equation}
and that any of the weak-limiting points of the 
sequence $(\hat u_{N,\e})\subset \X_{\bar u}$
is itself a minimizer of the restriction of $\F^\e$
to $\X_{\bar u}$. However, owing to 
Lemma~\ref{lem:uniq_loc_min}, we know that
 $\bar u$ is the 
unique minimizer of the restriction of $\F^\e$
to $\X_{\bar u}$. Therefore, we deduce that
$\hat u_{N,\e}\weak_{L^2} \bar u$ as $N\to\infty$.
We are left to show that the latter convergence 
holds also with respect to the strong topology of
$L^2$.
Using a similar reasoning as in the proof of
Corollary~\ref{cor:conv_min}, 
from \eqref{eq:conv_min_F_eps}
we obtain the identity
\begin{equation*}
\frac\beta2 ||\bar u||_{L^2}^2 = \lim_{N\to \infty}
\left( \frac\beta2||\hat u_{N,\e}||_{L^2}^2
+ \frac\e2 ||\hat u_{N,\e} - \bar u||_{L^2}^2 \right).
\end{equation*}
Finally, recalling the weak semi-continuity 
of the $L^2$-norm \eqref{eq:norm_weak_semicont}, the 
previous expression yields
\eqref{eq:strong_conv_u_N_eps}.
\end{proof}

We are now in position to prove the Maximum Principle
for the local minimizers of the ensemble optimal
control problem related to the functional
$\F:\U\to\R_+$.

\begin{theorem} \label{thm:ens_PMP}
Let us assume that the mappings 
$(x,\theta)\mapsto 
\frac{\partial}{\partial x}F_i(x,\theta)$ 
are continuous for every $i=0,\ldots,k$, as well as
the gradient $(x,\theta)\mapsto \nabla_x a(x,\theta)$.
Let $\bar u \in \U$ be a local minimizer of the 
functional $\F:\U\to\R_+$ introduced in 
\eqref{eq:fun_ens_endpoint}.
Let $X_{\bar u}:[0,1]\times \Theta \to \R^n$ be the 
mapping defined in \eqref{eq:def_evol_ens} that
collects the trajectories of the ensemble corresponding
to the control $\bar u$, and let us consider 
the application 
$\Lambda_{\bar u}:[0,1]\times\Theta\to
(\R^n)^*$ introduced in \eqref{eq:def_Lambda} that satisfies
\begin{equation} \label{eq:Lambda_PMP}
\begin{cases}
 \partial_t \Lambda_{\bar u}(t,\theta) = - \Lambda_{\bar u}(t,\theta)
\left( 
\frac{\partial F_0(X_{\bar u}(t,\theta), \theta)}{\partial x}
+ \sum_{i=1}^k{\bar u}_i(t) 
\frac{\partial F_i(X_{\bar u}(t,\theta), \theta)}{\partial x} 
\right) & \mbox{for a.e. } t\in [0,1],\\
\Lambda_{\bar u}(1,\theta) = \nabla_x a(X_{\bar u}(1,\theta),\theta),
\end{cases}
\end{equation}
for every $\theta \in \Theta$.
Then, we have that
\begin{equation}\label{eq:PMP_ens}
\bar u(t) \in 
\arg \max_{\!\!\!\!\!\!
\!\!\!\!\!\! v\in \R^k}
\left\{
\int_\Theta -\Lambda_{\bar u}(t,\theta)\cdot 
F(X_{\bar u}(t,\theta),\theta))\cdot v \, d\mu(\theta) 
- \frac{\beta}{2}|v|^2_2
\right\}
\end{equation}
for a.e. $t\in [0,1]$.
\end{theorem}
\begin{proof}
Let us fix $\e>0$ and, for every $N\geq 1$,
let us consider the functional 
$\F^{N,\e}:\X_{\bar u}\to\R_+$ and let
$\hat u_{N,\e} \in \arg \min_{\X_{\bar u}}\F^{N,\e}$.
As done in the proof of Theorem~\ref{thm:PMP_N},
the problem of minimizing $\F^{N,\e}$ over 
$\X_{\bar u}$ can be reduced to a classical
optimal control problem with end-point cost.
Therefore, using similar computations as in the
proof of Theorem~\ref{thm:PMP_N}, we deduce that
for every $N\geq 1$ the control $\hat u_{N,\e}$
is associated with a normal Pontryagin extremal of
the cost functional $\F^{N,\e}$.
Using the notations introduced in 
Remark~\ref{rmk:PMP_formulation}, 
if we consider the application 
$\Lambda_{\hat u_{N,\e}}:[0,1]\times\Theta\to
(\R^n)^*$  
defined in \eqref{eq:def_Lambda} and corresponding to
the admissible control $\hat u_{N,\e}\in \U$, 
we obtain that for a.e. $t\in [0,1]$
\begin{equation*}
\hat u_{N,\e}(t) \in 
\arg \max_{\!\!\!\!\!\!
\!\!\!\!\!\! v\in \R^k}
\left\{
\int_\Theta -\Lambda_{\hat u_{N,\e}}(t,\theta)\cdot 
F(X_{\hat u_{N,\e}}(t,\theta),\theta))\cdot v \, d\mu_N(\theta) 
- \frac{\beta}{2}|v|^2_2
- \frac{\e}{2}|v-\bar u(t)|_2^2
\right\},
\end{equation*}
i.e.,
\begin{equation}\label{eq:PMP_aux}
\hat u_{N,\e}(t) = 
\frac{1}{\beta + \e}\left(
\e \bar u(t) - 
\int_\Theta \left[
\Lambda_{\hat u_{N,\e}}(t,\theta)\cdot 
F(X_{\hat u_{N,\e}}(t,\theta),\theta))
\right]^T
 \, d\mu_N(\theta) 
\right)
\end{equation}
for a.e. $t\in [0,1]$ and for every $N\geq 1$.
For every $N\geq1$,
we denote by $Z_N\subset [0,1]$ the set of instants
with null Lebesgue measure
where the identity \eqref{eq:PMP_aux} does not hold.
In virtue of Proposition~\ref{prop:G_conv_F_eps},
we have that
$\hat u_{N,\e} \to_{L^2} \bar u$ as $N\to \infty$,
and, up to the extraction of a subsequence that we do 
not rename for simplicity, this implies that
there exists $Z_\infty\subset [0,1]$ with zero Lebesgue
measure such that 
$\hat u_{N,\e}(t) \to \bar u(t)$ as $N\to\infty$
for every $t\in [0,1]\setminus
Z_\infty$.
On the other hand, owing to Proposition~
\ref{prop:unif_conv_map_X} and Proposition~
\ref{prop:unif_conv_map_Lambda}, we deduce that
for every $t\in [0,1]$ the sequence of functions
$(f^t_N)_{N\geq1}$ satisfy $f^t_N\to_{C^0}f^t$
as $N\to\infty$, where
$f^t_N,f^t:\Theta\to\R^m$ are defined as follows:
\begin{align*}
&\theta\mapsto f^t_N(\theta) = \left[
\Lambda_{\hat u_{N,\e}}(t,\theta)\cdot 
F(X_{\hat u_{N,\e}}(t,\theta),\theta))
\right]^T, \\
&\theta\mapsto f^t(\theta) = \left[
\Lambda_{\bar u}(t,\theta)\cdot 
F(X_{\bar u}(t,\theta),\theta))
\right]^T.
\end{align*}
Moreover, recalling that $\mu_N\weak^* \mu$ as 
$N\to\infty$ by assumption, if we set
$Z:= Z_\infty \cup \bigcup_{N\geq 1}Z_N$, then 
for every
$t\in [0,1]\setminus Z$ we can take the pointwise
limit of \eqref{eq:PMP_aux} as 
$N\to\infty$, which yields:
\begin{equation*}
\bar u(t) = 
\frac{1}{\beta + \e}\left(
\e \bar u(t) - 
\int_\Theta \left[
\Lambda_{\bar u}(t,\theta)\cdot 
F(X_{\bar u}(t,\theta),\theta))
\right]^T
 \, d\mu(\theta) 
\right),
\end{equation*}
i.e.,
\begin{equation}\label{eq:PMP_limit}
\bar u(t) = -
\frac{1}{\beta}
\int_\Theta \left[
\Lambda_{\bar u}(t,\theta)\cdot 
F(X_{\bar u}(t,\theta),\theta))
\right]^T
 \, d\mu(\theta)
\end{equation}
for a.e. $t\in [0,1]$. From \eqref{eq:PMP_limit} - which we observe does not depend on the choice
of $\e>0$ - we finally obtain 
\eqref{eq:PMP_ens}.
\end{proof}

\begin{remark}
Theorem~\ref{thm:ens_PMP} shows that any local minimizer of the functional $\F$ is associated with a normal extremal of the ensemble optimal control problem. Moreover, we observe that there are no nontrivial abnormal extremals. Indeed, if we take $\epsilon \in \R$ and we consider $\Lambda_{\bar u}(1,\theta)=\epsilon \nabla_x a(X_{\bar u}(1,\theta),\theta)$ for every $\theta\in \Theta$ as the final-time datum for \eqref{eq:Lambda_PMP}, when $\epsilon = 0$ we obtain $(\Lambda_{\bar u},\epsilon)\equiv 0$. Finally, we observe that, in virtue of the concave quadratic term, the maximization problem \eqref{eq:PMP_ens} always admits a solution. Hence, there are no singular arcs. 
\end{remark}

\begin{remark} \label{rmk:PMP_G_conv}
For some \textit{global} minimizers $\bar u\in \U$ of the functional $\F:\U\to\R_+$ defined as in \eqref{eq:def_fun_ens}, Theorem~\ref{thm:ens_PMP} can be directly deduced from the $\Gamma$-convergence result established in Section~\ref{sec:G_conv}.
Namely, this is the case for those global minimizers $\bar u\in \arg\min_\U \F$ that can be recovered as the limiting points of the minimizers of the approximating functionals $\F^N:\U\to\R_+$ introduced in \eqref{eq:def_fun_N}.
Indeed, if $\hat u_N\in \arg\min_\U \F^N$ for every$N\geq 1$ and $\bar u\in \U$ is an $L^2$-strong accumulation point of the sequence $(\hat u_N)_{N\geq 1}$, then Corollary~\ref{cor:conv_min} guarantees that $\bar u\in \arg\min_\U \F$, and we can obtain the condition \eqref{eq:PMP_ens} by repeating the proof of Theorem~\ref{thm:ens_PMP} with $\e =0$. 

\noindent
However, in general, given a family of functionals $\mathcal I^N:\X\to\R$ on a metric space $(\X,d)$ such that $\mathcal I^N\to^\Gamma \mathcal I$ as $N\to\infty$, there could be elements in $\arg \min_\X \mathcal{I} $ that cannot be recovered as limiting points of minimizers of $(\mathcal{I}^N)_{N\geq 1}$. 
For instance, if we set $\X=[-1/2, 1/2]$ with the Euclidean distance, we have that the functions $\mathcal{I}^N : \X\to\R$ defined as $\mathcal{I}^N(x) := |x|^N$ are $\Gamma$-converging as $N\to\infty$ to the function $\mathcal{I}\equiv 0$.
On one hand, we have that $\arg \min_\X \mathcal{I}= \X$, while $\arg \min_\X \mathcal{I}^N = \{ 0 \}$ for every $N\in\NN$.
As a matter of fact, the minimizers of $\mathcal{I}$ in $\X \setminus \{ 0 \}$ cannot be recovered as a limit of minimizers of $(\mathcal{I}^N)_{N\in\NN}$.

\noindent
For this reason, in our case, the introduction of the auxiliary functionals $\F^{\e}$ and $(\F^{N,\e})$ in, respectively,  \eqref{eq:def_F_eps} and \eqref{eq:def_F_N_eps} is precisely aimed at managing this situation, as well as deducing the Maximum Principle also for \textit{local} minimizers, and not only for \textit{global} minimizers.
\end{remark}

\begin{remark} \label{rem:cont_PMP}
Results concerning the necessary optimality conditions for ensemble optimal control problems are of great interest from the theoretical viewpoint. A natural question is whether they could be successfully employed to derive numerical methods for the approximate resolutions of such problems. Some efforts in this direction were done in \cite{BCFH22}, where the authors obtain a \textit{mean-field} Maximum Principle for problems with uncertain initial datum and with the controlled dynamics unaffected by the unknown parameter. In that framework, a key-ingredient of the Maximum Principle \cite[Theorem~4.1]{BCFH22} is a real-valued function $\psi \in C^1([0,1], C^2_c(\R^n))$ that solves a backward-evolution PDE. We observe that the quantity $\nabla_x \psi$ is somehow related to the function $\Lambda_u$ that we introduced in our discussion (see \cite[Proposition~4.9]{BCFH22} for more details). In \cite{BCFH22} the authors proposed a numerical scheme for their mean-field optimal control problem relying on an approximated computation of the solution of the backward-evolution PDE. Despite the encouraging results obtained in the experiments, the main drawback of this approach is that the resolution of the PDE is affordable only in low dimensions (e.g., in \cite{BCFH22} examples in dimensions $1$ and $2$ were considered). 
\end{remark}

\end{section}

\begin{section}{Numerical schemes for optimal control 
of ensembles} \label{sec:num_schemes}

In the present section we introduce two numerical schemes for finite-ensemble optimal control problems with end-pint cost.
The starting points are the results of Section~\ref{sec:grad_flow}, and we follow an approach similar to \cite{S2}. The first method consists of the projection of the field $\G^N:\U\to\U$ induced by $\F^N$ onto a finite-dimensional subspace $\U_M\subset \U$.
The second one is based on the Pontryagin Maximum Principle and it was first proposed in \cite{SS80}.

Before proceeding, we introduce the notations and the framework that are shared by the two methods. Let us consider the interval $[0,1]$, i.e., the evolution time horizon of the ensemble of controlled dynamical systems \eqref{eq:ens_ctrl_sys}, and for $M\geq 2$ let us take the equispaced nodes $\{ 0, \frac1M,\ldots, \frac{M-1}M, 1 \}$. Recalling that $\U:=L^2([0,1],\R^k)$, let us define the subspace $\U_M\subset \U$ as follows:
\begin{equation} \label{eq:def_subspace_ctrls}
u\in \U_M \iff u(t) =
\begin{cases}
u_1& \mbox{if } 0\leq t< \frac1M\\
\vdots\\
u_M & \mbox{if } \frac{M-1}{M}\leq t\leq 1,
\end{cases}
\end{equation}
where $u_1,\ldots,u_M\in \R^k$.
For every $l=1,\ldots,M$, we shall write $u_l = (u_{1,l},\ldots,u_{k,l})$ to denote the components of $u_l\in\R^k$. Then, any element $u\in\U_M$ will be represented by the following array:
\begin{equation}\label{eq:not_u_scheme}
u = (u_{i,l})^{i=1,\ldots,k}_{l=1,\ldots,M}.
\end{equation}
For every $N\geq 1$, let $\mu_N$ be the discrete probability measure \eqref{eq:emp_meas} on $\Theta$ that approximates the probability measure $\mu$ involved in the definition of the functional $\F:\U\to\R_+$ in \eqref{eq:fun_ens_endpoint}.
Let $\{ \theta_1,\ldots,\theta_N \}\subset\Theta$ be the points charged by $\mu_N$, and, for every $j=1,\ldots,N$, let $x_u^{\theta_j}:[0,1]\to \R^n$ be the solution of \eqref{eq:ens_ctrl_sys} corresponding to the parameter $\theta_j$ and to the control $u$.
Then, for every $j=1,\ldots,N$ and $l=0,\ldots,M$ we define the array that collects the evaluation of the trajectories at the time nodes:
\begin{equation}\label{eq:not_x_scheme}
(x_l^j)^{j=1,\ldots,N}_{l=0,\ldots,M}, \qquad
x^j_l := x_u^{\theta_j}\left(\frac{l}{M}\right).
\end{equation}
We observe that in \eqref{eq:not_x_scheme} we
dropped the reference to the control that generates the trajectories. This is done to avoid hard notations, since we hope that it will be clear from the context the correspondence between trajectories and control. 
Similarly, for every $j=1,\ldots,N$, let $\lambda_u^j:[0,1]\to(\R^n)^*$ be the solution of \eqref{eq:adj_eq_N}, and let us introduce the corresponding array of the evaluations:
\begin{equation}\label{eq:not_lambda_scheme}
(\lambda_l^j)^{j=1,\ldots,N}_{l=0,\ldots,M},\qquad
\lambda_l^j:= \lambda_u^j\left( \frac{l}{M} \right).
\end{equation}

\begin{subsection}{Projected gradient field} \label{subsec:grad_method}
In this subsection we describe a method for the numerical minimization of the functional $\F^N:\U\to\R_+$ defined as in \eqref{eq:fun_ens_N_endpoint}. This algorithm consists of the projection of the gradient field $\G^N:\U\to\U$ derived in \eqref{eq:grad_field_N} onto the finite-dimensional subspace $\U_M\subset \U$ defined as in \eqref{eq:def_subspace_ctrls}.
This approach has been introduced in \cite{S2}, where it has been studied the problem of observations-based approximations of diffeomorphisms. 
We observe that we can explicitly compute the expression of the orthogonal projector $P_M:\U\to\U_M$. Indeed, we have
\begin{equation}\label{eq:orth_proj_M}
P_M[u](t) = \begin{cases}
M \int_{0}^{\frac1M} u(t)\,dt
& \mbox{if } 0\leq t<\frac1M,\\
\vdots \\
M \int_{\frac{M-1}{M}}^{1} u(t)\,dt
& \mbox{if } \frac{M-1}{M}\leq t\leq 1,
\end{cases}
\end{equation}  
for every $u\in\U$.
Thus, we can can define the projected field $\G^N_M:\U_M\to\U_M$ as 
\begin{equation}\label{eq:proj_field_M_N}
\G^N_M[u] := P_M[\G^N[u]]
\end{equation}
for every $u\in\U_M$, and we end up with a vector field on a finite-dimensional space.
At this point, in view of the numerical implementation of the method, it is relevant to observe that the computation of $\G^N[u]$ requires the knowledge of the trajectories $x_u^{\theta_1},\ldots,x_u^{\theta_N}:[0,1]\to\R^n$ and of the curves $\lambda_u^1,\ldots,\lambda_u^N :[0,1]\to(\R^n)^*$. However, during the execution of the algorithm, we have access only to the (approximated) values of these functions at the time nodes $\{ 0, \frac1M,\ldots, 1\}$.
Therefore, we need to adapt \eqref{eq:proj_field_M_N} to meet our needs. 
For every $u\in \U_M$, let us consider the corresponding arrays $(x_l^j)_{l=0,\ldots,M}^{j=1,\ldots,N}$ and $(\lambda^j_l)_{l=0,\ldots,M}^{j=1,\ldots,N}$ defined as in \eqref{eq:not_x_scheme} and \eqref{eq:not_lambda_scheme}, respectively.
In practice, they can be computed using standard numerical schemes for the approximation of ODEs. For every $l=1,\ldots,M$, we use the approximation \begin{align*}
M\int_{\frac{l-1}{M}}^{\frac{l}{M}}  
\sum_{j=1}^N\alpha_j \Big(F^{\theta_j}&(x_u^{\theta_j}(t))^T\cdot
\lambda_u^j(t)^T\Big) + \beta u(t) 
\,dt \\
&\simeq 
\frac12
\sum_{j=1}^N \alpha_j\left(F^{\theta_j}(x_{l-1}^j)^T \cdot 
\lambda_{l-1}^{j\,T} 
+  F^{\theta_j}(x_{l}^j)^T\cdot
\lambda_{l}^{j\,T}
\right)+ \beta u_l,
\end{align*}
where $\alpha_1,\ldots,\alpha_N$ are the coefficients of convex combination involved in the definition of $\mu_N = \sum_{j=1}^N\alpha_j\delta_{\theta_j}$.
Then, for every $u\in \U_M$,
after computing the corresponding arrays
$(x_l^j)_{l=0,\ldots,M}^{j=1,\ldots,N}$
and $(\lambda^j_l)_{l=0,\ldots,M}^{j=1,\ldots,N}$
with a proper ODEs integrator scheme,
we use the quantity
 $\Delta u =(\Delta u_1,\ldots, \Delta u_M)
  \in \U_M$ to approximate
$\G^N_M[u]$, where we set
\begin{equation}\label{eq:grad_field_M_N_approx}
\Delta u_l :=  \frac12
\sum_{j=1}^N \alpha_j\left(F^{\theta_j}(x_{l-1}^j)^T\cdot
\lambda_{l-1}^{j\,T} 
+  F^{\theta_j}(x_{l}^j)^T\cdot
\lambda_{l}^{j\,T}
\right)+ \beta u_l
\end{equation}
for every $l=1,\ldots,M$.
We are now in position to describe the 
Projected Gradient Field algorithm.
We report it in Algorithm~\ref{alg:proj_grad_field}.

\begin{algorithm}
\KwData{ 
\begin{itemize}
\item 
$\{\theta_1,\ldots\theta_N\}\subset \Theta$
subset of parameters;
\item
$F_0^{\theta_1},\ldots,F_0^{\theta_N}:\R^n\to\R^n$
drift fields;
\item
$F^{\theta_1},\ldots,F^{\theta_N}:\R^n\to
\R^{n\times k}$ controlled fields;
\item
$(x_0^j)^{j=1,\ldots,N}=(x_0^{\theta_1},\ldots,x_0^{
\theta_N})$ initial states of trajectories;
\item 
$a(\cdot,\theta_1),\ldots,a(\cdot,\theta_N):\R^n\to\R_+$ end-point costs, and $\beta>0$.
\end{itemize}
{\bf Algorithm setting:} 
$M= \mathrm{dim}\,\, \U_M$,
$\tau\in (0,1)$, $c \in (0,1)$, 
$\gamma >0$, $\max_{\mathrm{iter}}\geq 1$,
$u\in \U_M$.}

$h\gets\frac{1}{M}$\;

\For(\tcp*[f]{First computation of  
trajectories}){$j=1,\ldots,N$ }{
		Compute $(x^j_l)_{l=1,\ldots,M}$ using
		$(u_l)_{l=1,\ldots,M}$ and $x_0^j$\;
	}

$\mathrm{Cost}\gets \sum_{j=1}^N \alpha_j
a(x^j_M,\theta_j) + \frac\beta2 ||u||_{L^2}^2$\;
$\mathrm{flag}\gets 1$\;
\For(\tcp*[f]{Iterations of Projected Gradient Field}){$r=1,\ldots,\max_{\mathrm{iter}}$ }{
	\If(\tcp*[f]{Update covectors only if necessary}){$\mathrm{flag}=1$ }{
		\For(\tcp*[f]{Backward computation of 
		covectors }){$j=1,\ldots,N$ }{
			$\lambda^j_M\gets 
			 \nabla a(x^j_M,\theta_j)$\;
			Compute $(\lambda_l^j)_{l=0,\ldots,M-1}$
			using $(u_l)_{l=1,\ldots,M}$, 
			$(x^j_l)_{l=0,\ldots,M}$ and 
			$\lambda_M^j$\;					
			}
		}
	\For(\tcp*[f]{Compute $\Delta u$ using \eqref{eq:grad_field_M_N_approx}}){$l=1,\ldots,M$ 
	}{
		$\Delta u_l \gets  \frac12
\sum_{j=1}^N\alpha_j \left(F^{\theta_j}(x_{l-1}^j)^T\cdot
\lambda_{l-1}^{j\,T} 
+  F^{\theta_j}(x_{l}^j)^T\cdot
\lambda_{l}^{j\,T}
\right)+ \beta u_l$\;
		}
		$u^{\mathrm{new}}\gets u - \gamma \Delta u$\;
		\For(\tcp*[f]{Forward computation
		of trajectories}){$j=1,\ldots,N$ }{
			$x^{j,\mathrm{new}}_0\gets x^j_0$\;			
					Compute $(x^{j,\mathrm{new}}_l)_{l=1,\ldots,M}$ 	using 
		$(u^{\mathrm{new}}_l)_{l=1,\ldots,M}$ and $x_0^{j,\mathrm{new}}$\;		
			}
	$\mathrm{Cost^{new}}
\gets  \sum_{j=1}^N \alpha_j
a(x^{j,\mathrm{new}}_M ,\theta_j) + \frac\beta2 ||u^{\mathrm{new}}||_{L^2}^2$\;	
	\eIf(\tcp*[f]{Backtracking for $\gamma$}){$\mathrm{Cost}\geq \mathrm{Cost^{new}}
+ c\gamma || \Delta u ||_{L^2}^2$ }{
		$u\gets u^{\mathrm{new}}$,
		$x\gets x^{\mathrm{new}}$\;
		$\mathrm{Cost}\gets \mathrm{Cost^{new}}$\;
		$\mathrm{flag} \gets 1$\;
		}
		{
		$\gamma \gets \tau \gamma$\;
		$\mathrm{flag} \gets 0$\;
		}
	}
\caption{Projected Gradient Field}
\label{alg:proj_grad_field}
\end{algorithm}

\begin{remark}\label{rmk:par_proj_grad}
We observe that the {\it for loops} at the lines
9--12 and 18--21 (corresponding, respectively, to
the update of the curves of
covectors and of the trajectories) can be carried 
out in parallel with respect to the 
index $j=1,\ldots,N$. This can be considered 
when dealing with large sub-ensembles of 
parameters.
\end{remark}

\begin{remark}\label{rmk:backtracking}
The step-size $\gamma>0$ for 
Algorithm~\ref{alg:proj_grad_field} is set 
during the initialization of the method, and it
is adaptively adjusted through the 
{\it if clause} at the lines 23--30
via the classical Armijo-Goldstein condition
(see, e.g., \cite[Section~1.2.3]{N18}).
We observe that, if the update of the control
at the $r$-th iteration is rejected, 
at the $r+1$-th iteration 
it is not necessary to re-compute the 
array of covectors 
$(\lambda_l^j)^{j=1,\ldots,N}_{l=0,\ldots,M}$.
In this regards, the {\it if clause} at the line
8 prevents this computation in the case
of rejection at the previous passage.
\end{remark}

\end{subsection}

\begin{subsection}{Iterative Maximum Principle}
\label{subsec:iter_PMP}
In this subsection we present a second numerical method for the minimization of the functional $\F^N:\U\to\R_+$, based on the Pontryagin Maximum Principle. 
The idea of using the Maximum Principle to design approximation schemes for optimal control problems was well established in the Russian literature (see \cite{CL} for a survey paper in English). 
Here we adapt to our problem the method proposed in \cite{SS80}, which is in turn a stabilization of one of the algorithms reported in \cite{CL}. Finally, this approach has been recently followed in \cite{S2} in the framework of diffeomorphisms approximation.

The key idea relies on iterative updates of the control through the resolution of a maximization problem related to the condition \eqref{eq:max_cond}. 
However, the substantial difference from Algorithm~\ref{alg:proj_grad_field} consists in the fact that the controls and the trajectories are computed simultaneously.
More precisely, let us consider $M\geq 1$ and let $\U_M\subset\U$ be the finite-dimensional subspace introduced in \eqref{eq:def_subspace_ctrls}.
Given an initial guess $u=(u_l)_{l=1,\ldots,M}\in \U_M$, let $(x_l^j)_{l=0,\ldots,M}^{j=1,\ldots,N}$ and $(\lambda^j_l)_{l=0,\ldots,M}^{j=1,\ldots,N}$ be the corresponding arrays, defined as in \eqref{eq:not_x_scheme} and \eqref{eq:not_lambda_scheme}, respectively. 
For $l=1$, the value of $u^{\mathrm{new}}_1$ (i.e., the updated value of control in the time interval $[0,1/M]$) is computed using $(x_0^j)^{j=1,\ldots,N}$ and $(\lambda^j_0)^{j=1,\ldots,N}$ as follows:
\begin{equation}\label{eq:PMP_update}
u_1^{\mathrm{new}} =
\arg \max_{\!\!\!\!\!\!
\!\!\!\!\!\! v\in \R^k}
\left\{
\sum_{j=1}^N \alpha_j\left(
-\lambda^j_0 \cdot 
F^{\theta_j}(x_0^j)\cdot v \right) 
- \frac{\beta}{2}|v|^2_2
-\frac{1}{2\gamma}|v-u_1|^2_2
\right\},
\end{equation}
where $\gamma>0$ plays the role of the
step-size of the update, and $\alpha_1,\ldots,\alpha_N$ are the coefficients of convex combination involved in the definition of $\mu_N = \sum_{j=1}^N\alpha_j\delta_{\theta_j}$.
From the value $u_1^{\mathrm{new}}$ just 
obtained and the initial conditions
$(x_0^j)^{j=1,\ldots,N}$, we compute 
$(x_1^j)^{j=1,\ldots,N}$, i.e.,
the approximation of the trajectories at the 
time-node $1/M$. At this point,
using $(x_1^j)^{j=1,\ldots,N}$ and 
$(\lambda_1^j)^{j=1,\ldots,N}$,
we calculate $u_2^{\mathrm{new}}$ 
with a maximization problem analogue to
\eqref{eq:PMP_update}. Finally, we 
sequentially repeat the
same procedure for every $l=2,\ldots,M$.
We report the scheme in 
Algorithm~\ref{alg:iter_PMP}.

\begin{algorithm}
\KwData{ 
\begin{itemize}
\item 
$\{\theta_1,\ldots\theta_N\}\subset \Theta$
subset of parameters;
\item
$F_0^{\theta_1},\ldots,F_0^{\theta_N}:\R^n\to\R^n$
drift fields;
\item
$F^{\theta_1},\ldots,F^{\theta_N}:\R^n\to
\R^{n\times k}$ controlled fields;
\item
$(x_0^j)^{j=1,\ldots,N}=(x_0^{\theta_1},\ldots,x_0^{
\theta_N})$ initial states of trajectories;
\item 
$a(\cdot,\theta_1),\ldots,a(\cdot,\theta_N):\R^n\to\R_+$ end-point costs, and $\beta>0$.
\end{itemize}
{\bf Algorithm setting:} 
$M= \mathrm{dim}\,\, \U_M$,
$\tau\in (0,1)$, 
$\gamma >0$, $\max_{\mathrm{iter}}\geq 1$,
$u\in \U_M$.}

$h\gets\frac{1}{M}$\;

\For(\tcp*[f]{First computation of  
trajectories}){$j=1,\ldots,N$ }{
		Compute $(x^j_l)_{l=1,\ldots,M}$ using
		$(u_l)_{l=1,\ldots,M}$ and $x_0^j$\;
	}

$\mathrm{Cost}\gets \sum_{j=1}^N \alpha_j
a(x^j_M,\theta_j) + \frac\beta2 ||u||_{L^2}^2$\;
$\mathrm{flag}\gets 1$\;
\For(\tcp*[f]{Iterations of Iterative Maximum Principle}){$r=1,\ldots,\max_{\mathrm{iter}}$ }{
	\If(\tcp*[f]{Update covectors only if necessary}){$\mathrm{flag}=1$ }{
		\For(\tcp*[f]{Backward computation of 
		covectors }){$j=1,\ldots,N$ }{
			$\lambda^j_M\gets 
			\nabla a(x^j_M,\theta_j)$\;
			Compute $(\lambda_l^j)_{l=0,\ldots,M-1}$
			using $(u_l)_{l=1,\ldots,M}$, 
			$(x^j_l)_{l=0,\ldots,M}$ and 
			$\lambda_M^j$\;					
			}
		}
		
	$(x_0^{j,\mathrm{new}})^{j=1,\ldots,N}\gets 
	(x_0^{j})^{j=1,\ldots,N}$\;
	$(\lambda_0^{j,\mathrm{corr}}
	)^{j=1,\ldots,N}\gets 
	(\lambda_0^{j})^{j=1,\ldots,N}$\;
	\For(\tcp*[f]{Update of controls and trajectories}){$l=1,\ldots,M$ 
	}{

		$u_l^{\mathrm{new}} \gets
\arg \max_{ v\in \R^k}
\left\{
\sum_{j=1}^N \alpha_j \left(
-\lambda^{j,\mathrm{corr}}_{l-1}\cdot 
F^{\theta_j}(x_{l-1}^{j,\mathrm{new}})\cdot v \right) 
- \frac{\beta}{2}|v|^2_2
-\frac{1}{2\gamma}|v-u_l|^2_2
\right\}$\;
	\For{$j=1,\ldots,N$}{
	Compute
	$x_l^{j,\mathrm{new}}$ using 
	$x_{l-1}^{j,\mathrm{new}}$ and 
	$u_l^{\mathrm{new}}$\;
	$\lambda_l^{j, \mathrm{corr}}
	\gets \lambda_l^j - 
	\alpha_j \nabla a(x_l^j,\theta_j) 
	+ \alpha_j \nabla a(x_l^{j,\mathrm{new}},
	\theta_j)$\;}
		}

	$\mathrm{Cost^{new}}
\gets  \sum_{j=1}^N\alpha_j
a(x^{j,\mathrm{new}}_M ,\theta_j) + \frac\beta2 ||u^{\mathrm{new}}||_{L^2}^2$\;	
	\eIf(\tcp*[f]{Backtracking for $\gamma$}){$\mathrm{Cost}> \mathrm{Cost^{new}}$ }{
		$u\gets u^{\mathrm{new}}$,
		$x\gets x^{\mathrm{new}}$\;
		$\mathrm{Cost}\gets \mathrm{Cost^{new}}$\;
		$\mathrm{flag} \gets 1$\;
		}
		{
		$\gamma \gets \tau \gamma$\;
		$\mathrm{flag} \gets 0$\;
		}
	}
\caption{Iterative Maximum Principle}
\label{alg:iter_PMP}
\end{algorithm}

\begin{remark} \label{rmk:max_comput}
The maximization at line 17 can be solved directly
at a very low computational cost. 
Indeed, we have that
\begin{equation*}
u_l^{\mathrm{new}}\gets 
\frac1{1+\gamma\beta}\left(
u_l - \sum_{j=1}^N \alpha_j \left(
\lambda_l^{j,\mathrm{corr}} \cdot
F^{\theta_j}
(x_{l-1}^{j,\mathrm{new}})
\right)^T
\right)
\end{equation*}
for every $l=1,\ldots,M$.
This is 
essentially due to the fact that the 
systems of the ensemble \eqref{eq:ens_ctrl_sys}
have an affine dependence on the control. 
\end{remark}

\begin{remark}\label{rmk:par_PMP}
As well as in Algorithm~\ref{alg:proj_grad_field},
in this case the computation of 
$(\lambda_l^j)_{l=0,\ldots,M-1}^{j=1,\ldots,N}$
can be carried out in parallel 
(see the {\it for loop} at the 
lines 9--12).
Unfortunately, this is no more true 
for the update of the trajectories, since
in Algorithm~\ref{alg:iter_PMP} the 
computation of 
$(x_l^{j,\mathrm{new}})^{j=1,\ldots,N}$
takes place immediately after obtaining
$u_l^{\mathrm{new}}$, for every
$l=1,\ldots,M$ (see lines 17--21).
\end{remark}

\begin{remark}\label{rmk:corr_PMP}
At the line 20 of Algorithm~\ref{alg:iter_PMP}
we introduce a correction for the value of
the covector. 
This feature is not present in the original
scheme proposed in \cite{SS80}, where the 
authors considered optimal control problems
without end-point cost.
\end{remark}

\begin{remark}\label{rmk:backtrack_PMP}
Also in Algorithm~\ref{alg:iter_PMP}
 the step-size is adaptively
adjust, and it is reduced if, after the iteration, 
the value of the functional has not decreased.
In case of rejection of the update, it is not
necessary to recompute 
$(\lambda_l^j)_{l=0,\ldots,M}^{j=1,\ldots,N}$.
This is a common feature with 
Algorithm~\ref{alg:proj_grad_field},
as observed in 
Remark~\ref{rmk:backtracking}.  
\end{remark}

\end{subsection}
\end{section}

\begin{section}{Numerical experiments}\label{sec:num_exp}
In this section we test the algorithms 
described in Section~\ref{sec:num_schemes}
on an optimal control problem involving an
ensemble of linear dynamical systems in
$\R^2$. Namely, given 
$\theta_{\min} < \theta_{\max} \in \R$,
let us set $\Theta:=[\theta_{\min},\theta_{\max}]
\subset\R$, and let us consider the 
ensemble of control systems
\begin{equation}\label{eq:ens_ctrl_num_exp}
\begin{cases}
\dot x^\theta_u(t) = A^\theta x^\theta_u(t) 
+ b_1 u_1(t) + b_2 u_2(t)
& \mbox{a.e. in }[0,1],\\
x^\theta_u(0) = x_0^\theta,
\end{cases}
\end{equation}
where 
$\theta\mapsto x_0^\theta$ is a continuous function
that prescribes the initial states,
$u=(u_1,u_2)^T\in \U:= L^2([0,1],\R^2)$,
and, for every $\theta\in \Theta$, we have
\begin{equation}\label{eq:data_ens_num}
A^\theta :=\left(
\begin{matrix}
0 & 1\\
\theta & 0
\end{matrix}
\right),\quad 
b_1 := \left(\begin{matrix}
1 \\
0
\end{matrix}
\right), \quad
b_2 := \left(\begin{matrix}
0 \\
1
\end{matrix}
\right).
\end{equation}
For every $N\geq 1$ and for every subset of 
parameters
$\{ \theta_1,\ldots,\theta_N \}\subset\Theta$, 
we represent the corresponding sub-ensemble
of \eqref{eq:ens_ctrl_num_exp} as an
affine-control system on $\R^{2N}$, as
done in Section~\ref{sec:grad_flow}.
More precisely, we consider
\begin{equation}\label{eq:ens_ctrl_exp_N}
\begin{cases}
\dot \xx_u (t) = \mathbf{A}^N\xx_u(t) +\mathbf{b}_1u_1(t)
+ \mathbf{b}_2u_2(t)
& \mbox{a.e. in }[0,1],\\
\xx_u(0) = \xx_0,
\end{cases}
\end{equation}
where $\mathbf{A}^N\in \R^{2N\times2N}$ and
$\mathbf{b}_1,\mathbf{b}_2\in \R^{2N}$ are 
defined as follows:
\begin{equation} \label{eq:data_ens_num_N}
\mathbf{A}^N := \left(
\begin{matrix}
A^{\theta_1} & \mathbf{0}_{2\times 2}&\ddots \\
\mathbf{0}_{2\times 2} & \ddots&\mathbf{0}_{2\times 2}\\
\ddots&\mathbf{0}_{2\times 2}&A^{\theta_N}
\end{matrix}
\right), \,\,
\mathbf{b}_1 := 
\left(
\begin{matrix}
b_1\\
\vdots\\
b_1
\end{matrix}
\right),  \,\,
\mathbf{b}_2 := 
\left(
\begin{matrix}
b_2\\
\vdots\\
b_2
\end{matrix}
\right).
\end{equation}
Moreover, 
we observe that \eqref{eq:ens_ctrl_num_exp} can be
interpreted as a control system in the space
$C^0(\Theta,\R^2)$. Indeed, we can consider the
control system
\begin{equation}\label{eq:ens_ctrl_Banach}
X_{u, t}
 = X_0 + \int_0^t \mathcal{A}[X_{u,\tau}] \, d\tau
+ \int_0^t \mathfrak{b}_1u_1(\tau)
+ \mathfrak{b}_2u_2(\tau)\, d\tau, \quad
t\in[0,1],
\end{equation} 
where  
$\mathcal{A}:C^0(\Theta,\R^2)\to C^0(\Theta,\R^2)$
is the bounded linear operator defined as 
\begin{equation*}
\mathcal{A}[Y](\theta) := A^\theta Y(\theta)
\end{equation*}
for every $\theta\in \Theta$ and for every
$Y\in C^0([0,1],\R^2)$, and 
$\mathfrak{b}_1,\mathfrak{b}_1:\Theta \to \R^2$ are 
defined as
\begin{equation*}
\mathfrak{b}_1(\theta) := b_1,\quad 
\mathfrak{b}_2(\theta) := b_2
\end{equation*}
for every $\theta\in \Theta$, and finally
$X_0:\Theta\to\R^2$ satisfies
$X_0(\theta):= x_0^\theta$ for every 
$\theta\in \Theta$.
The integrals in \eqref{eq:ens_ctrl_Banach}
should be understood in the Bochner sense,
and, for every $u\in \U$,
 the existence and uniqueness of
a continuous curve $t\mapsto X_{u,t}$ in
$C^0(\Theta,\R^2)$ solving \eqref{eq:ens_ctrl_Banach}
descends from classical
results in linear inhomogeneous 
ODEs in Banach spaces (see, e.g., 
\cite[Chapter~3]{DK74}).
In particular, from the uniqueness we deduce that
\begin{equation}\label{eq:sol_Bochn}
X_{u,t}(\theta) = x_u^\theta(t)
\end{equation}
for every $u\in\U$, $t\in [0,1]$ and 
$\theta\in\Theta$, where 
$x_u^\theta:[0,1]\to\R^2$ is the solution of
\eqref{eq:ens_ctrl_num_exp} corresponding to
the parameter $\theta$ and to the control $u$.
We now prove some controllability results for the 
control systems \eqref{eq:ens_ctrl_exp_N} and
\eqref{eq:ens_ctrl_Banach}.

\begin{proposition}\label{prop:ctrl_exp}
For every $N\geq 1$ and for every subset
$\{ \theta_1,\ldots,\theta_N \}\subset \Theta$,
let us consider 
$\mathbf{y}_{\mathrm{tar}}\in \R^{2N}$. Then,
there exists a control $\bar u\in \U$
such that the corresponding solution  
$\xx_{\bar u}:[0,1]\to\R^{2N}$
of \eqref{eq:ens_ctrl_exp_N} satisfies
$\xx_{\bar u}(1) = \mathbf{y}_{\mathrm{tar}}$.\\
Moreover, for every 
$Y_{\mathrm{tar}}\in C^0(\Theta,\R^2)$
and for every $\e>0$, there exists a control
$u_\e\in \U$ such that
the curve $t\mapsto X_{u_\e, t}$ that solves
\eqref{eq:ens_ctrl_Banach} satisfies
\[
||Y-X_{u_\e, 1}||_{C^0}\leq \e.
\]
\end{proposition}

\begin{proof}
We observe that the first part of the thesis 
follows if we prove the exact controllability 
of the system \eqref{eq:ens_ctrl_exp_N}.
An elementary result in control theorey (see, e.g.,
\cite[Theorem~3.3]{AS}) ensures that the last
condition is implied by the identity
\begin{equation*} 
\mathrm{span}\left\{
(\mathbf{A}^N)^r\, \mathbf{b_1}, (\mathbf{A}^N)^r\,\mathbf{b_2}|\,\, 0\leq r\leq 2N-1
\right\} = \R^{2N}.
\end{equation*}
A direct computation shows that this is actually the 
case.

As regards the second part of the thesis, 
owing to \cite[Theorem~3.1.1]{T75}
we have that it is sufficient to prove that
\begin{equation} \label{eq:cond_control_banach}
\overline{
\mathrm{span}\left\{\mathcal{A}^r[\mathfrak{b}_1],
\mathcal{A}^r[\mathfrak{b}_2]| \,\, r\geq 0
 \right\}}^{C^0} = C^0(\Theta,\R^2).
\end{equation}
We observe that 
\[
\mathrm{span}\left\{\mathcal{A}^r[\mathfrak{b}_1],
\mathcal{A}^r[\mathfrak{b}_2]| \,\, r\geq 0
 \right\}
= 
\mathrm{span}\left\{
\left(
\begin{matrix}
\theta^r\\
0
\end{matrix}
\right),
\left(
\begin{matrix}
0\\
\theta^r
\end{matrix}
\right)
| \,\, r\geq 0
 \right\}, 
\]
therefore the identity \eqref{eq:cond_control_banach}
follows from the Weierstrass Theorem
on polynomial approximation.
\end{proof}

We now introduce the problem that we studied in the
numerical simulations. 
We set $\theta_{\min }=-\frac12,
\theta_{\max}=\frac12$, and we consider 
on $\Theta=[-\frac12,\frac12]$ the probability
measure $\mu$, distributed as a 
$\mathrm{Beta}(4,4)$ centered at $0$.
We observe that during the experiments we assumed to have no explicit knowledge of the probability measure $\mu$. On the other hand, we imagined to be able to sample observations from that distribution, and we pursued the \textit{data driven} approach described in Remark~\ref{rem:discr_prob_meas}.  
After that the approximated optimal control had been computed, we validated the policy just obtained on a testing sub-ensemble of \textit{newly-sampled} parameters. 
Let us assume that
the initial data in \eqref{eq:ens_ctrl_num_exp}
is not affected by the parameter $\theta$, i.e,
there exists $x_0\in \R^2$ such that 
$x_0^\theta = x_0$ for every $\theta\in\Theta$.
We imagine that we want to steer the
end-points of the trajectories
of \eqref{eq:ens_ctrl_num_exp} as close as
possible to a target point $y_{\mathrm{tar}}\in\R^2$. 
Therefore, we consider the functional
$\F:\U\to\R_+$ defined as 
\begin{equation}\label{eq:def_fun_ens_exp}
\F(u) := \int_\Theta |x_u^\theta(1) -y_{\mathrm{tar}} 
|_2^2 \, d\mu(\theta) + \frac{\beta}{2}||u||_{L^2}^2
\end{equation}
for every $u\in \U$.
We observe that the second part of 
Proposition~\ref{prop:ctrl_exp} implies that
we are in the situation described in 
Remark~\ref{rmk:role_beta}.
Indeed, if we set $Y_{\mathrm{tar}}(\theta) := 
y_{\mathrm{tar}}$ for every $\theta\in \Theta$,
we have that for every $\e>0$ there exists
$u_\e\in\U$ such that  
\begin{equation*}
\int_\Theta |x_{u_\e}^\theta(1) -y_{\mathrm{tar}} 
|_2^2 \, d\mu(\theta)
\leq ||X_{u_\e,1} -Y_{\mathrm{tar}}||_{C^0}
\leq \frac\e2,
\end{equation*}
where we used the identity \eqref{eq:sol_Bochn}.
Therefore, in correspondence of small values
of $\beta$, we expect that the minimizers of 
\eqref{eq:def_fun_ens_exp} drive the end-point 
of the controlled trajectories very close to
$y_{\mathrm{tar}}$. In the simulations we 
considered $\beta = 10^{-3}$.
Finally, we approximated the probability measure
$\mu$ with the empirical distribution $\mu_N$,
obtained with $N$ independent samplings of $\mu$,
using $N = 300$.
Moreover, we chose $x_0=(0,0)^T$ and 
$y_{\mathrm{tar}} = (-1,-1)^T$.
We report below the results obtained with 
Algorithm~\ref{alg:proj_grad_field}
and  Algorithm~\ref{alg:iter_PMP}, where
we set $M=64$. 
We observed that performances of the two 
numerical methods are very similar, as regards
both the qualitative aspect of the controlled 
trajectories and the decay of the cost during the
execution.

\begin{figure}
\centering
\includegraphics[scale=0.47]{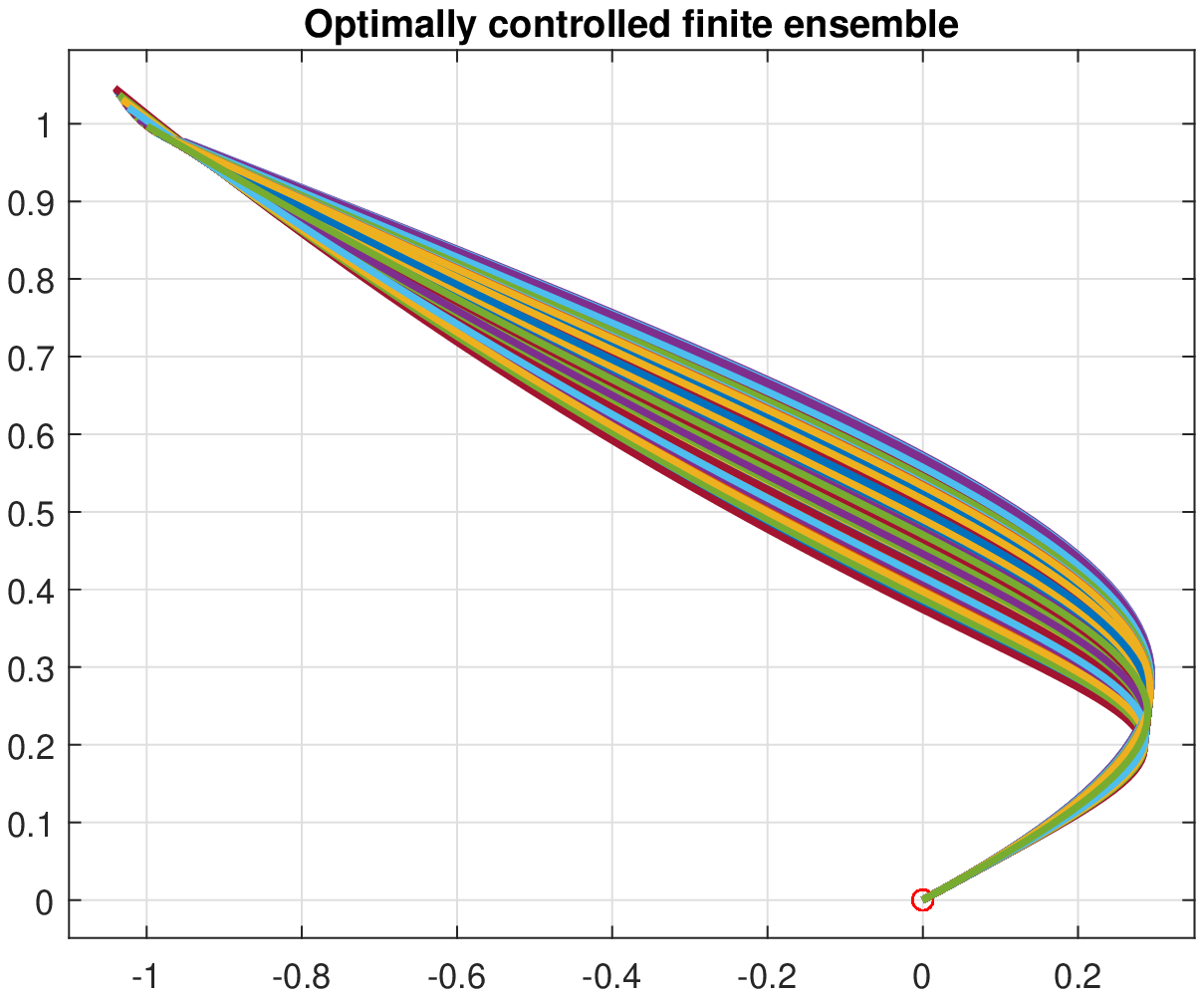}
\includegraphics[scale=0.47]{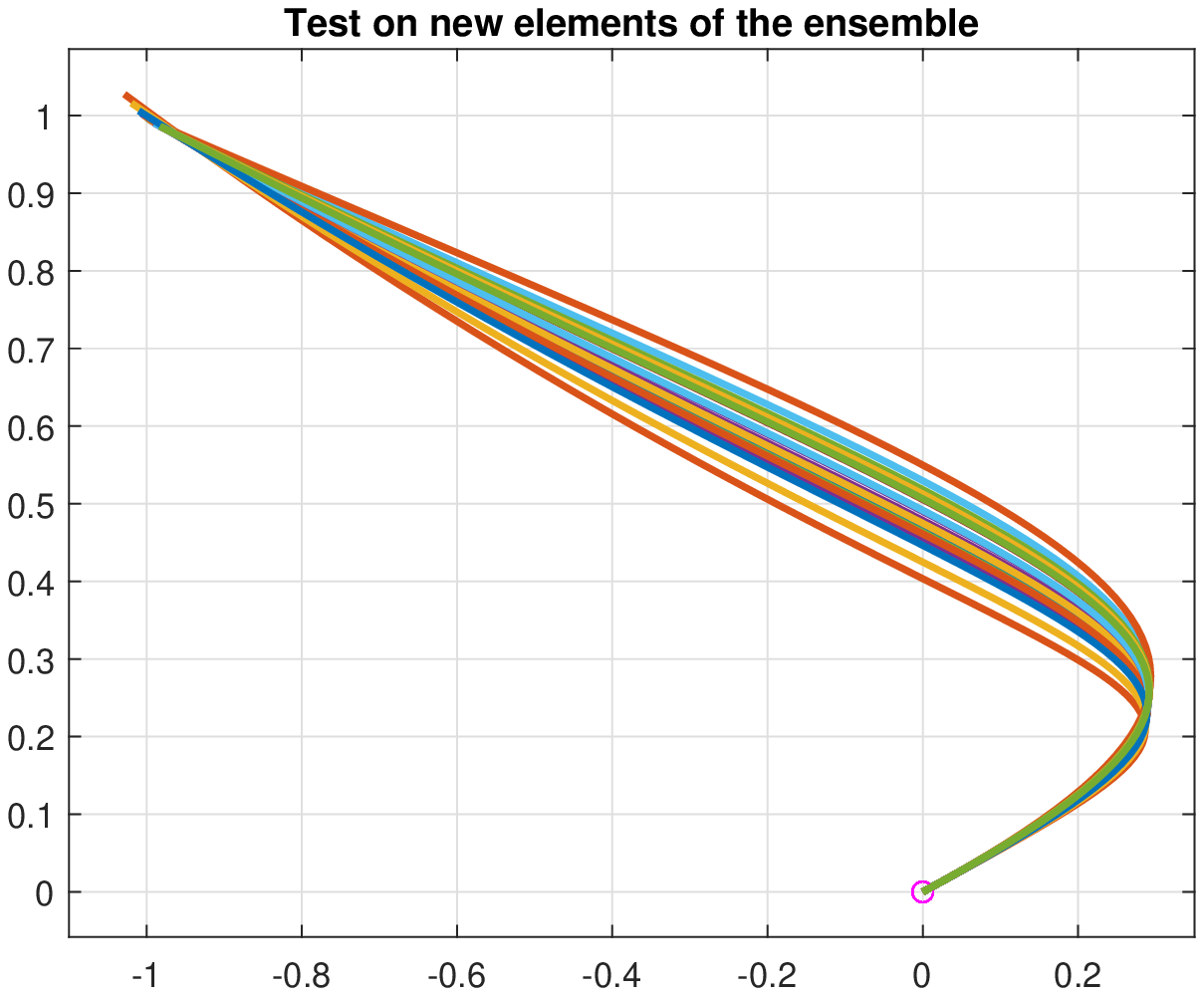}
\caption{Controlled ensemble. On the left, we 
reported the optimally controlled trajectories
of the sub-ensemble of $\Theta$ obtained by 
sampling $N=300$ parameters.
On the right, we tested the controls obtained 
before on a new sub-ensemble of $\Theta$, obtained
by sampling $20$ new parameters. As we can see, 
the trajectories belonging to
 the testing sub-ensemble
 are correctly steered to the 
target point $y_{\mathrm{tar}}=(-1,1)^T$.}
\end{figure}

\begin{figure}
\centering
\includegraphics[scale=0.47]{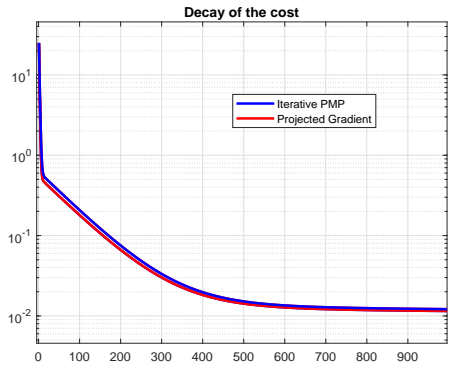}
\caption{In the graph we reported the decay of the 
discrete cost achieved by
 Algorithm~\ref{alg:proj_grad_field} 
(Projected Gradient)
 and Algorithm~\ref{alg:iter_PMP} (Iterative PMP).
As we can see, the performances on this problem
are very similar.}
\end{figure}

\end{section}

\section*{Conclusions}
In this paper we considered the problem of the
optimal control of an ensemble of affine-control
systems. We proved the well posedness of the 
corresponding minimization problem, and we 
showed with a $\Gamma$-convergence argument 
how we can reduce the original problem to
an approximated one, involving ensembles with
a finite number of elements. For these ones, 
in the case of end-point cost,
we proposed two numerical schemes for the 
approximation of the optimal control.
We finally tested the methods on a 
ensemble optimal control problem in dimension two.

For future development, we plan to study algorithms
also for more general costs, and not only for
terminal-state penalization.
Moreover, we hope to extend the $\Gamma$-convergence results to some proper class of ensembles of nonlinear-control systems. 
As well as in the affine-control case, we expect that weak topologies on the space of controls are required to have equi-coercivity of the functionals. On the other hand, the challenging aspect is that, in nonlinear-control systems, weakly convergent controls do not induce, in general, locally $C^0$-strongly convergent flows.

\appendix

\section{Auxiliary results of Subsection~\ref{subsec:mapp_Xu}}
\label{app:X_u}
Here we prove
some auxiliary properties of the mapping
$X_u:[0,1]\times \Theta\to\R^n$, 
which has been defined in
\eqref{eq:def_evol_ens} for every $u\in\U$.
Before proceeding, we recall a version of the
Gr\"onwall-Bellman inequality.

\begin{lemma}[Gr\"onwall-Bellman Inequality] 
\label{lem:Gron}
Let $f:[a,b]\to\R_+$ be a
non-negative continuous function
and let us assume that there exists
a constant $\alpha>0$ and a non-negative 
function $\beta\in L^1([a,b],\R_+)$
such that
\[
f(s) \leq \alpha + \int_a^s\beta(\tau)f(\tau) \,d\tau
\]
for every $s\in[a,b]$. Then, for every 
$s\in[a,b]$ the following inequality holds:
\begin{equation} \label{eq:Gron_ineq}
f(s) \leq \alpha e^{||\beta||_{L^1}}.
\end{equation}
\end{lemma}
\begin{proof}
This result follows directly from
\cite[Theorem~5.1]{EK86}. 
\end{proof}

We first prove that for every
$u\in\U$ the mapping $X_u:[0,1]\times\Theta
\to \R^n$ is bounded.

\begin{lemma} \label{lem:bound_X}
For every $u\in\U$,
let $X_u:[0,1]\times\Theta \to \R^n$ be the
application defined in \eqref{eq:def_evol_ens}
collecting the trajectories of the ensemble of
control systems \eqref{eq:ens_ctrl_sys}.
Then, for every $R>0$ there exists 
$C_R>0$ such that, if $||u||_{L^2}\leq R$,
we have 
\begin{equation} \label{eq:bound_X}
|X_u(t,\theta)|_2\leq C_R,
\end{equation}
for  every $(t,\theta)\in[0,1]\times \Theta$.
\end{lemma}

\begin{proof}
Using \eqref{eq:def_evol_ens},
in virtue of the sub-linear growth
inequalities 
\eqref{eq:sub_lin_F_0}-\eqref{eq:sub_lin_F},
we observe 
that
%\begin{equation*}
%X_u(t,\theta) = x_u^\theta(t)=
%x^\theta_0 + \int_0^t \left(
%F_0^\theta(\xth_u(\tau)) + F^\theta(\xth_u(\tau))u(\tau)
%\right)\, d\tau.
%\end{equation*}
%we have that
\begin{align*}
|X_u(t,\theta)|_2=| x^\theta_u(t) |_2
&\leq | x^\theta_0 |_2
+ \int_0^t C(
1 +  |x^\theta_u(\tau)|_2
) (1 + |u(\tau)|_1) \, d\tau\\
&\leq | x^\theta_0 |_2 + C(1 + \sqrt k ||u||_{L^2})
+ \int_0^t C(1 + |u(\tau)|_1)|x^\theta_u(\tau)|_2
\,d\tau
\end{align*}
for every $\theta\in\Theta$ and $t\in[0,1]$.
Using Lemma~\ref{lem:Gron}, we deduce that
\begin{equation*}
|X(t,\theta)|_2 \leq 
\Big(
| x^\theta_0 |_2 + C(1 + \sqrt k ||u||_{L^2})
\Big)
e^{C\left(1+\sqrt k ||u||_{L^2}\right)}
\end{equation*}
for every $(t,\theta)\in[0,1]\times\Theta$.
Recalling the boundedness of $\theta\mapsto
\xth_0$ provided by
\eqref{eq:bound_ens_init_val}, the thesis follows
from the last inequality.
\end{proof}

We shall prove that, when the control $u$
varies in a bounded subset of $\U$,
the corresponding
 functions $X_u:[0,1]\times\Theta\to\R^n$ that captures the evolution
of the ensemble of control systems 
\eqref{eq:ens_ctrl_sys} are uniformly equi-continuous 
on their domain.
We first show \textit{separately}
the uniform equi-continuity for 
the variables in the time domain $[0,1]$ and in the
parameter domain $\Theta$. 
In the next result we observe that the 
trajectories of the ensemble are 
H\"older-continuous, uniformly with respect
to the parameter $\theta\in\Theta$.

\begin{lemma}\label{lem:unif_holder_time}
For every $u\in\U$,
let $X_u:[0,1]\times\Theta \to \R^n$ be the
application defined in \eqref{eq:def_evol_ens}
collecting the trajectories of the ensemble of
control systems \eqref{eq:ens_ctrl_sys}.
Then, for every $R>0$ there exists 
$L_R>0$ such that, if $||u||_{L^2}\leq R$, then
\begin{equation} \label{eq:unif_holder_time}
|X_u(t_1,\theta)-X_u(t_2,\theta)|_2
\leq L_R |t_1-t_2|^{\frac12}
\end{equation}
for every $t_1,t_2\in[0,1]$ and for every
$\theta\in\Theta$.
\end{lemma}
\begin{proof}
Owing to Proposition~\ref{prop:sob_Holder}
and recalling that
$X_u(t,\theta)=\xth_u(t)$ for every $(t,\theta)
\in [0,1]\times\Theta$
by \eqref{eq:def_evol_ens}, we observe that
the thesis follows if we prove that
there exists a bounded 
subset of $H^1$ 
that includes the trajectories
$\{ x_u^\theta:[0,1]\to\R^n \}_{\theta\in\Theta}$
of \eqref{eq:ens_ctrl_sys}
for every admissible control $u\in\U$
satisfying $||u||_{L^2}\leq R$.

From Lemma~\ref{lem:bound_X} we obtain that
for every $R>0$ there exists
$C_R>0$ such that
\begin{equation} \label{eq:point_bound_traj}
|x_u^\theta(t)|_2\leq C_R
\end{equation}
for every $t\in[0,1]$ and for every
$u\in\U$ such that $||u||_{L^2}\leq R$. 
In virtue
of Lemma~\ref{lem:bound_X} and 
the sub-linear inequalities 
\eqref{eq:sub_lin_F_0}-\eqref{eq:sub_lin_F}, 
we deduce that
 for every $R>0$ there exists
$C_R'>0$ such that
\begin{equation*}
\sup_{\theta\in\Theta}|F_0^\theta(x_u^\theta(t))|_2\leq 
C_R', \qquad \sup_{\theta\in\Theta}
\,\sup_{i=1,\ldots,k}|F^\theta_i(x_u^\theta(t))|_2
\leq C_R'
\end{equation*}
for every $t\in[0,1]$ and for every
$u\in\U$ such that $||u||_{L^2}\leq R$.
Therefore, we have that
\begin{equation} \label{eq:point_bound_sob_der_traj}
|\dot x^\theta_u(t)|_2  \leq C'_R(1+|u(t)|_1)
\end{equation}
for every $t\in[0,1]$, for every
$\theta \in \Theta$ and for every
$u\in\U$ such that $||u||_{L^2}\leq R$. 
Combining \eqref{eq:point_bound_sob_der_traj} 
and \eqref{eq:point_bound_traj}, we deduce that
there exists $C''_R>0$ such that
\begin{equation*}
||x_u^\theta||_{H^1} \leq C''_R
\end{equation*} 
for every $\theta\in\Theta$ and for every
$u\in\U$ such that $||u||_{L^2}\leq R$.
The last inequality and 
Proposition~\ref{prop:sob_Holder} imply that
\begin{equation*}
|\xth_u(t_1) -\xth_u(t_2)|_2
\leq L_R|t_1-t_2|^{\frac12}
\end{equation*}
for every $t_1,t_2\in[0,1]$, for every
$\theta\in\Theta$ and for every
$u\in\U$ such that $||u||_{L^2}\leq R$, 
where we set $L_R:=\sqrt{C_R''}$.
This establishes \eqref{eq:unif_holder_time}.
\end{proof}

Before proceeding, we introduce the modulus of
continuity of the function
$x_0:\Theta\to\R^n$ defined in 
\eqref{eq:def_x0}.
Indeed, since $x_0:\Theta\to\R^n$ is a continuous
function defined on a compact domain, it is 
uniformly continuous, i.e., there exists
a non-decreasing function 
$\omega:\R_+\to\R_+$ satisfying
$0=\omega(0) = \lim_{r\to 0^+}\omega(r)$
and such that
\begin{equation} \label{eq:unif_cont_x0}
|x_0(\theta_1)-x_0(\theta_1)|_2 \leq 
\omega(|\theta_1-\theta_2|_2)
\end{equation}
for every $\theta_1,\theta_2\in\Theta$.

\begin{lemma} \label{lem:unif_cont_Theta_X}
For every $u\in\U$,
let $X_u:[0,1]\times\Theta \to \R^n$ be the
application defined in \eqref{eq:def_evol_ens}
collecting the trajectories of the ensemble of
control systems \eqref{eq:ens_ctrl_sys}.
Then, for every $R>0$ there exists 
$\omega_R:\R_+\to\R_+$
 such that, if $||u||_{L^2}\leq R$, then
\begin{equation} \label{eq:unif_cont_Theta_X}
|X_u(t,\theta_1)-X_u(t,\theta_2)|_2
\leq \omega_R(|\theta_1-\theta_2|_2)
\end{equation}
for every $t\in[0,1]$ and for every 
$ \theta_1,\theta_2 \in \Theta$, where
$\omega_R$ is a non-decreasing function
that satisfies 
$\omega(0)= \lim_{r\to0^+}\omega_R(r)=0$.
\end{lemma}

\begin{proof}
Recalling  \eqref{eq:def_evol_ens}, we compute
\begin{align*}
|X_u(t,\theta_1)-X_u(t,\theta_2)|_2 &=
|x_u^{\theta_1}(t) - x_u^{\theta_2}(t)|_2\\
&\leq |x_0(\theta_1)-x_0(\theta_2)|_2
+ \int_0^t |F_0^{\theta_1}(x_u^{\theta_1}(\tau))
-F_0^{\theta_2}(x_u^{\theta_2}(\tau))|_2 \,d\tau\\
&\qquad +
\int_0^t \sum_{i=1}^k \left(
|F^{\theta_1}_i(x_u^{\theta_1}(\tau))
-F^{\theta_2}_i(x_u^{\theta_2}(\tau))|_2 
|u_i(\tau)| \right)
d\tau
\end{align*}
for every $t\in[0,1]$, for every $\theta_1,\theta_2
\in\Theta$ and for every $u\in\U$.
Using \eqref{eq:ens_fields} and the
Lipschitz-continuity conditions 
\eqref{eq:lips_F_0}-\eqref{eq:lips_F}, the last
expression yields
\begin{align*}
|x_u^{\theta_1}(t)-x_u^{\theta_2}(t)|_2
&\leq |x_0(\theta_1)-x_0(\theta_2)|_2\\
& \qquad
+ \int_0^t L(1+|u(\tau)|_1)\left(|x_u^{\theta_1}(\tau)-
x_u^{\theta_2}(\tau)|_2 + |\theta_1-\theta_2|_2
\right) d\tau\\
&\leq |x_0(\theta_1)-x_0(\theta_2)|_2
+L(1+\sqrt{k}||u||_{L^2})|\theta_1-\theta_2|_2\\
&\qquad + \int_0^t L
(1+|u(\tau)|_1)(|x_u^{\theta_1}(\tau)-
x_u^{\theta_2}(\tau)|_2) \,d\tau
\end{align*} 
for every $t\in [0,1]$, for every
$\theta_1,\theta_2\in\Theta$ and for
every $u\in\U$.
Owing to Lemma~\ref{lem:Gron}, from the
last inequality we deduce that
%\begin{equation} \label{eq:2t_tr_ineq_unif_X}
%|X_u(t,\theta_1)-X_u(t,\theta_2)|_2 =
%|x_u^{\theta_1}(t)-x_u^{\theta_2}(t)|_2
%\leq \omega_R(|\theta_1-\theta_2|_2),
%\end{equation}
\eqref{eq:unif_cont_Theta_X} holds
for every $t\in[0,1]$, for every 
$\theta_1,\theta_2\in\Theta$ and for
every $u\in \U$ with $||u||_{L^2}\leq R$,
where the
function $\omega_R:\R_+\to\R_+$ is defined as follows:
\begin{equation*}
\omega_R(r):= e^{L\left(1 + \sqrt k R\right)}
\left( \omega(r) + L(1+\sqrt{k} R )r
\right),
\end{equation*} 
and $\omega:\R_+\to\R_+$ is a modulus of
continuity for the mapping $x_0:\Theta\to\R^n$
(see \eqref{eq:unif_cont_x0}).
\end{proof}

We are now in position of stating the uniform
equi-continuity result.

\begin{lemma} \label{lem:unif_cont_X}
For every $u\in\U$,
let $X_u:[0,1]\times\Theta \to \R^n$ be the
application defined in \eqref{eq:def_evol_ens}
collecting the trajectories of the ensemble of
control systems \eqref{eq:ens_ctrl_sys}.
Then, for every $R>0$ there exists 
$L_R>0$ and $\omega_R:\R_+\to\R_+$
 such that, if $||u||_{L^2}\leq R$, then
\begin{equation} \label{eq:unif_cont_X}
|X_u(t_1,\theta_1)-X_u(t_2,\theta_2)|_2
\leq L_R|t_1-t_2|^{\frac12} +
\omega_R(|\theta_1-\theta_2|_2)
\end{equation}
for every $(t_1,\theta_1),(t_2,\theta_2)
\in [0,1]\times \Theta$, where
$\omega_R$ is a non-decreasing function
satisfying
$\omega(0)= \lim_{r\to0^+}\omega_R(r)=0$.
\end{lemma}

\begin{proof}
The thesis \eqref{eq:unif_cont_X} follows directly
from the triangular inequality and from
Lemma~\ref{lem:unif_holder_time} and
Lemma~\ref{lem:unif_cont_Theta_X}.
\end{proof}

\section{Auxiliary results of Subsection~\ref{subsec:Lambda}}

Here we establish some auxiliary results concerning the mapping $\Lambda_u$ defined in \eqref{eq:def_Lambda}.
We use the same scheme used in Appendix~\ref{app:X_u}, and we first show that $\Lambda_u$ is bounded.

\begin{lemma}
\label{lem:bound_Lambda}
Let us assume that the mappings 
$(x,\theta)\mapsto 
\frac{\partial}{\partial x}F_i(x,\theta)$ 
are continuous for every $i=0,\ldots,k$, as well as
the gradient $(x,\theta)\mapsto \nabla_x a(x,\theta)$.
For every $u\in\U$,
let $\Lambda_u:[0,1]\times\Theta \to (\R^n)^*$ be the
application defined in \eqref{eq:def_Lambda}.
Then, for every $R>0$ there exists 
$C_R>0$ such that, if $||u||_{L^2}\leq R$,
we have 
\begin{equation} \label{eq:bound_Lambda}
|\Lambda_u(t,\theta)|_2\leq C_R,
\end{equation}
for  every $(t,\theta)\in[0,1]\times \Theta$.
\end{lemma}
\begin{proof}
We preliminarily observe that, 
owing to Lemma~\ref{lem:bound_X} and the continuity of 
$\nabla_x a$, there exists 
$C_R^1>0$ such that
\begin{equation*}
\sup_{\theta\in \Theta}|\nabla_x a(X_u(t,\theta),\theta)|_2\leq C_R^1.
\end{equation*}
From the definition of $\Lambda_u:[0,1]\times\Theta
\to (\R^n)^*$ in \eqref{eq:def_Lambda},
it follows that
\begin{equation*}
\begin{split}
|\Lambda_u(t,\theta)|_2 = |\lambda_u^\theta(t)|_2
&\leq |\lambda_u^\theta(1)|_2 + \int_t^1
|\lambda_u^\theta(\tau)|_2 \left|
\frac{\partial F_0(x_u^\theta(\tau), \theta)}{\partial x}
+ \sum_{i=1}^ku_i(\tau) 
\frac{\partial F_i(x_u^\theta(\tau), \theta)}{\partial x}
\right|_2\, d\tau\\
& \leq
|\lambda_u^\theta(1)|_2 + \int_t^1
|\lambda_u^\theta(\tau)|_2 L(1+|u(\tau)|_1)\, d\tau
\end{split}
\end{equation*}
for every $(t,\theta)\in [0,1]\times \Theta$, where
we used \eqref{eq:lips_F_0}-\eqref{eq:lips_F} in the
last passage. Finally, combining
the previous inequality with Lemma~\ref{lem:Gron} and
\eqref{eq:norm_ineq}, we deduce \eqref{eq:bound_Lambda}. 
\end{proof}

In the next lemma we show that $\Lambda_u$ is H\"older-continuous in time.

\begin{lemma}\label{lem:unif_holder_time_Lambda}
Let us assume that the mappings 
$(x,\theta)\mapsto 
\frac{\partial}{\partial x}F_i(x,\theta)$ 
are continuous for every $i=0,\ldots,k$, as well as
the gradient $(x,\theta)\mapsto \nabla_x a(x,\theta)$.
For every $u\in\U$,
let $\Lambda_u:[0,1]\times\Theta \to (\R^n)^*$ be the
application defined in \eqref{eq:def_Lambda}.
Then, for every $R>0$ there exists 
$L_R>0$ such that, if $||u||_{L^2}\leq R$, then
\begin{equation} \label{eq:unif_holder_time_Lambda}
|\Lambda_u(t_1,\theta)-\Lambda_u(t_2,\theta)|_2
\leq L_R |t_1-t_2|^{\frac12}
\end{equation}
for every $t_1,t_2\in[0,1]$ and for every
$\theta\in\Theta$.
\end{lemma}

\begin{proof}
We recall that by the definition \eqref{eq:def_Lambda}
we have 
$\Lambda_u(\cdot,\theta) = \lambda_u^\theta(\cdot)$,
for every $\theta\in \Theta$, where 
$\lambda_u^\theta:[0,1]\to(\R^n)^*$ solves 
the linear differential equation \eqref{eq:def_lambda_theta_u}.
Therefore, 
we employ the same strategy as in the proof of
Lemma~\ref{lem:unif_holder_time}, i.e., we show
that there exists a bounded 
subset of $H^1$ 
that includes the family of curves
$\{ \lambda_u^\theta:[0,1]\to(\R^n)^* \}_{\theta\in\Theta}$
for every admissible control $u\in\U$
satisfying $||u||_{L^2}\leq R$.
From Lemma~\ref{lem:bound_Lambda} it descends that
there exists $C_R>0$ such that
\begin{equation} \label{eq:bound_lambda_H1_1}
|\lambda_u^\theta(t)|_2 = |\Lambda_u(t,\theta)|_2
\leq C_R
\end{equation}
for every $t\in[0,1]$ and $\theta\in \Theta$.
On the other hand, we compute
\begin{equation} \label{eq:bound_lambda_H1_2}
\begin{split}
|\dot \lambda_u^\theta(t)|_2 &\leq
|\lambda_u^\theta(t)|_2 
\left|
\frac{\partial F_0(x_u^\theta(t), \theta)}{\partial x}
+ \sum_{i=1}^ku_i(t) 
\frac{\partial F_i(x_u^\theta(t), \theta)}{\partial x}
\right|_2\\
& \leq C_R L (1 + |u(t)|_1)
\end{split}
\end{equation}
for a.e. $t\in [0,1]$ and for every 
$\theta \in \Theta$.
Hence, combining \eqref{eq:bound_lambda_H1_1}-\eqref{eq:bound_lambda_H1_2} with
\eqref{eq:norm_ineq} and 
Proposition~\ref{prop:sob_Holder}, we deduce
\eqref{eq:unif_holder_time_Lambda}.
\end{proof}

In the following result we prove the uniform continuity of $\Lambda_u:[0,1]\times \Theta\to (\R^n)^*$ with respect to the second variable.

\begin{lemma} \label{lem:unif_cont_Theta_Lambda}
Let us assume that the mappings 
$(x,\theta)\mapsto 
\frac{\partial}{\partial x}F_i(x,\theta)$ 
are continuous for every $i=0,\ldots,k$, as well as
the gradient $(x,\theta)\mapsto \nabla_x a(x,\theta)$.
For every $u\in\U$,
let $\Lambda_u:[0,1]\times\Theta \to (\R^n)^*$ be the
application defined in \eqref{eq:def_Lambda}.
Then, for every $R>0$ there exists 
$\omega_R:\R_+\to\R_+$
 such that, if $||u||_{L^2}\leq R$, then
\begin{equation} \label{eq:unif_cont_Theta_Lambda}
|\Lambda_u(t,\theta_1)- \Lambda_u (t,\theta_2)|_2
\leq \omega_R(|\theta_1-\theta_2|_2)
\end{equation}
for every $t\in[0,1]$ and for every 
$ \theta_1,\theta_2 \in \Theta$, where
$\omega_R$ is a non-decreasing function
that satisfies 
$\omega(0)= \lim_{r\to0^+}\omega_R(r)=0$.
\end{lemma}

\begin{proof}
From the definition \eqref{eq:def_Lambda} and from
\eqref{eq:def_lambda_theta_u}, it follows that
\begin{equation} \label{eq:unif_cont_Lambda_1}
\begin{split}
|\Lambda_u(t,\theta_1) &- \Lambda_u (t,\theta_2)|_2
= 
|\lambda_u^{\theta_1}(t) - \lambda_u^{\theta_2}(t)|_2\\
& \leq |\nabla_x a (x_u^{\theta_1}(1),\theta_1)
- \nabla_x a (x_u^{\theta_2}(1),\theta_2)|_2 
\\
&\quad +\int_t^1 
|\lambda_u^{\theta_2}(\tau)|_2
\left(
\left|
\frac{\partial F_0(x_u^{\theta_1}(\tau), \theta_1)}{\partial x}
-
\frac{\partial F_0(x_u^{\theta_2}(\tau), \theta_2)}{\partial x}
\right|_2 \right.
\\
& \qquad\qquad\qquad\qquad\quad \left.
+ \sum_{i=1}^k |u_i(\tau)|
\left|
\frac{\partial F_i(x_u^{\theta_1}(\tau), \theta_1)}{\partial x}
-
\frac{\partial F_i(x_u^{\theta_2}(\tau), \theta_2)}{\partial x}
\right|_2
\right)
\,d\tau
\\
&\quad +\int_t^1 
|\lambda_u^{\theta_1}(\tau) - \lambda_u^{\theta_2}(\tau)|_2 \left|
\frac{\partial F_0(x_u^{\theta_1}(\tau), \theta_1)}{\partial x}
+ \sum_{i=1}^ku_i(\tau) 
\frac{\partial F_i(x_u^{\theta_1}(\tau), \theta_1)}{\partial x}
\right|_2\, d\tau
\end{split}
\end{equation}
for every $t\in [0,1]$ and for every 
$\theta_1,\theta_2 \in \Theta$.
In virtue of Lemma~\ref{lem:bound_X}, there exists
a compact set $K_R \subset \R^n$ such that
the image $X_u([0,1],\Theta) \subset K_R$ for every
$u\in \U$ with $||u||_{L^2}\leq R$. 
The continuity assumptions guarantee that 
$(x,\theta)\mapsto 
\frac{\partial}{\partial x}F_i(x,\theta)$ for $ i=0,\ldots,k$ 
and $(x,\theta)\mapsto \nabla_x a(x,\theta)$
are uniformly continuous when restricted to
$K_R \times \Theta$.
Moreover, in virtue of 
Lemma~\ref{lem:unif_cont_Theta_X}, we deduce that
the applications defined as
$\theta \mapsto 
\frac{\partial}{\partial x}F_i(X_u(t,\theta),\theta)$ for $ i=0,\ldots,k$
and  
$\theta \mapsto \nabla_x a(X_u(1,\theta),\theta)$
are uniformly equi-continuous for every choice of
$t\in [0,1]$ and $u\in \U$ with $||u||_{L^2}\leq R$.
Let $\omega'_R:\R_+\to\R_+$ be a modulus of
continuity for all these functions.
Hence, using Lemma~\ref{lem:bound_Lambda},
from \eqref{eq:unif_cont_Lambda_1} 
we obtain that there exists $C_R>0$ such that
\begin{equation*}
\begin{split}
|\Lambda_u(t,\theta_1)- \Lambda_u (t,\theta_2)|_2
&= 
|\lambda_u^{\theta_1}(t) - \lambda_u^{\theta_2}(t)|_2\\
& \leq \omega'_R(|\theta_1-\theta_2|_2)
+ C_R(1 + ||u||_{L^1})\omega'_R(|\theta_1-\theta_2|_2)
\\
& \quad +C_R\int_t^1
|\lambda_u^{\theta_1}(\tau) - \lambda_u^{\theta_2}(\tau)|_2 \left(
1
+ \sum_{i=1}^k|u_i(\tau)| 
\right)
\, d\tau
\end{split}
\end{equation*}
for every $t\in[0,1]$, for every $\theta_1,\theta_2
\in \Theta$ and for every $u\in \U$ with 
$||u||_{L^2}\leq R$.
Then, the thesis \eqref{eq:unif_cont_Theta_Lambda}
follows directly from Lemma~\ref{lem:Gron}.
\end{proof}

Finally, the next results proves the uniform continuity of $\Lambda_u$.

\begin{lemma} \label{lem:unif_cont_Lambda}
Let us assume that the mappings 
$(x,\theta)\mapsto 
\frac{\partial}{\partial x}F_i(x,\theta)$ 
are continuous for every $i=0,\ldots,k$, as well as
the gradient $(x,\theta)\mapsto \nabla_x a(x,\theta)$.
For every $u\in\U$,
let $\Lambda_u:[0,1]\times\Theta \to (\R^n)^*$ be the
application defined in \eqref{eq:def_Lambda}.
Then, for every $R>0$ there exists 
$L_R>0$ and $\omega_R:\R_+\to\R_+$
 such that, if $||u||_{L^2}\leq R$, then
\begin{equation} \label{eq:unif_cont_Lambda}
|\Lambda_u(t_1,\theta_1)-\Lambda_u(t_2,\theta_2)|_2
\leq L_R|t_1-t_2|^{\frac12} +
\omega_R(|\theta_1-\theta_2|_2)
\end{equation}
for every $(t_1,\theta_1),(t_2,\theta_2)
\in [0,1]\times \Theta$, where
$\omega_R$ is a non-decreasing function
satisfying
$\omega(0)= \lim_{r\to0^+}\omega_R(r)=0$.
\end{lemma}

\begin{proof}
The thesis \eqref{eq:unif_cont_Lambda} follows directly
from the triangular inequality and from
Lemma~\ref{lem:unif_holder_time_Lambda} and
Lemma~\ref{lem:unif_cont_Theta_Lambda}.
\end{proof}

\vspace{0.3cm}

\noindent
(A. Scagliotti).

\vspace{5pt} 

\noindent
\textsc{``School of Computation, Information and Technology'',\\ TU Munich, Garching b. M\"unchen, Germany.}
\vspace{3pt}

\noindent
\textsc{Munich Center for Machine Learning (MCML), Germany.}

\vspace{5pt} 
\noindent
\textit{Email address:} \texttt{scag -at- ma.tum.de}

\end{document}